\documentclass[twoside,11pt]{article}

\usepackage{jmlr2e}
\usepackage{microtype}
\usepackage{graphicx}
\usepackage{subfigure}
\usepackage{booktabs} % for professional tables
\usepackage{xcolor}
\usepackage{hyperref}

\usepackage[utf8]{inputenc} % allow utf-8 input
\usepackage[T1]{fontenc}    % use 8-bit T1 fonts
\usepackage{hyperref}       % hyperlinks
\usepackage{url}            % simple URL typesetting
\usepackage{booktabs}       % professional-quality tables
\usepackage{amsfonts}       % blackboard math symbols
\usepackage{nicefrac}       % compact symbols for 1/2, etc.
\usepackage{microtype}      % microtypography
\usepackage{amsmath}
\usepackage{algorithm,algorithmic}
\usepackage{multirow}
\usepackage{multicol}
\usepackage{natbib}
\usepackage{diagbox}
\usepackage{mathtools}
\newtheorem{assumption}{Assumption}

\usepackage{enumitem}
\usepackage{textcase,booktabs}
\usepackage{amsmath,amsfonts,bm}

\def\1{\bm{1}}

\DeclareMathAlphabet{\mathsfit}{\encodingdefault}{\sfdefault}{m}{sl}
\SetMathAlphabet{\mathsfit}{bold}{\encodingdefault}{\sfdefault}{bx}{n}

\newcommand{\E}{\mathbb{E}}

\newcommand{\R}{\mathbb{R}}

\DeclareMathOperator*{\argmin}{arg\,min}

\newcommand{\bx}{{\mathbf x}}
\newcommand{\by}{{\mathbf y}}
\renewcommand{\u}{\mathbf{u}}
\renewcommand{\v}{\mathbf{v}}
\newcommand{\bz}{{\mathbf z}}

\newcommand{\bw}{{\mathbf w}}
\newcommand{\X}{{\mathcal X}}
\newcommand{\Y}{{\mathcal Y}}
\newcommand{\Z}{{\mathcal Z}}

\newcommand{\bxi}{\boldsymbol{\xi}}
\newcommand{\bzeta}{\boldsymbol{\zeta}}

\newcommand{\bzt}{\boldsymbol{\zeta}}

\newcommand{\reals}{\mathbb{R}}

\def \x {\mathbf{x}}
\def \y {\mathbf{y}}
\def \z {\mathbf{z}}
\def \w {\mathbf{w}}
\def \F {\mathcal{F}}
\def \zu {\z_*}

\def \wh {\widehat{\w}}
\def \zh {\widehat{\z}}
\newcommand{\bu}{{\mathbf u}}
\newcommand{\bv}{{\mathbf v}}

\firstpageno{1}
\usepackage{lastpage}
\jmlrheading{22}{2021}{1-\pageref{LastPage}}{5/20; Revised
3/21}{5/21}{20-533}{Mingrui Liu, Hassan Rafique, Qihang Lin, Tianbao Yang}
\ShortHeadings{First-order Convergence for WCWC Min-max Problems}{Liu, Rafique, Lin, Yang}
\begin{document}

\title{First-order Convergence Theory for Weakly-Convex-Weakly-Concave Min-max Problems}

\author{\name Mingrui Liu \email mingruiliu.ml@gmail.com \\
    %   \addr Department of Computer Science\\
    %   George Mason University\\
    %   Fairfax, VA, 22030, USA\\
       \addr  Department of Computer Science\\
       The University of Iowa\\
   Iowa City, IA, 52242, USA
       \AND
       \name Hassan Rafique \email hassan-rafique@uiowa.edu \\
       \addr Department of Mathematics\\
       The University of Iowa\\
       Iowa City, IA, 52242, USA
   	\AND
   \name Qihang Lin \email qihang-lin@uiowa.edu \\
   \addr Business Analytics Department\\
   The University of Iowa\\
   Iowa City, IA, 52242, USA
   \AND
   \name Tianbao Yang \email tianbao-yang@uiowa.edu \\
   \addr Department of Computer Science\\
   The University of Iowa\\
   Iowa City, IA, 52242, USA}

\editor{Zhihua Zhang}

\maketitle

\begin{abstract}%   <- trailing '%' for backward compatibility of .sty file
In this paper, we consider first-order convergence theory and algorithms for solving a class of non-convex non-concave min-max saddle-point problems, whose objective function is weakly convex in the variables of minimization and weakly concave in the variables of maximization. 
It  has many important applications in machine learning including training Generative Adversarial Nets (GANs). We propose an algorithmic framework  motivated by the inexact proximal point method, where the weakly monotone variational inequality (VI) corresponding to the original min-max problem is solved through approximately solving a sequence of strongly monotone VIs constructed by adding a strongly monotone mapping to the original gradient mapping. 
We prove first-order convergence to  a nearly stationary  solution of the original min-max problem of the generic algorithmic framework and establish different rates by employing different algorithms for solving each strongly monotone VI. 
%Our algorithm generates a sequence of solutions that provably converges to a nearly stationary  solution of the original min-max problem. 
%Our algorithm has different complexities when the strongly monotone VIs are solved by different methods under different assumptions, including subgradient method, stochastic subgradient method for a bounded mapping,  gradient descent method, extragradient method, Nesterov's accelerated method and variance-reduction methods for a Lipschitz continuous operator.  
Experiments verify the convergence theory and also demonstrate the effectiveness of the proposed methods on training GANs. 
\end{abstract}


\section{Introduction}
This paper is motivated by solving the following min-max saddle-point problem:
\begin{align}\label{eqn:P}
\min_{\x\in\X}\max_{\y\in\Y}  f(\x, \y),
\end{align}
where $\X\subset\mathbb{R}^p$ and $\Y\subset\mathbb{R}^q$ are closed convex sets, and $f(\x, \y)$ is real-valued, continuous, not necessarily convex in $\x$, and not necessarily concave in $\y$. This problem has broad applications in machine learning and statistics. One such example that recently attracts tremendous attention in machine learning is training Generative Adversarial Networks (GAN) where $\x$ denotes the parameter of the generator network and $\y$ denotes the parameter of the discriminator network~\citep{Goodfellow:2014:GAN:2969033.2969125}. 

Problem \eqref{eqn:P} has been well studied when $f(\x, \y)$ is convex in $\x$ and concave in $\y$ and many algorithms have been developed with provable non-asymptotic convergence properties~\citep{journals/sovnemir,nemirovski-2005-prox,nemirovski2009robust,Chambolle:2011:FPA:1968993.1969036,DBLP:journals/ml/YangMJZ15}. However,  when $f(\x, \y)$ is non-convex in $\x$ and non-concave in $\y$, the developments of algorithms for~(\ref{eqn:P}) with provable non-asymptotic convergence guarantees remain rare and most existing studies under this setting only focus on asymptotic convergence analysis~\citep{doi:10.1137/15M1026924,DBLP:journals/corr/abs-1711-00141,DBLP:conf/nips/HeuselRUNH17,DBLP:conf/nips/NagarajanK17,Solodov1999ANP,Wang2001,bao05pr,DBLP:journals/corr/corr1609}. Besides non-asymptotic convergence, an important property of an algorithm for \eqref{eqn:P} is what type of solutions it can guarantee. In the studies of convex-concave saddle-point problems, one is interested in finding a \emph{saddle point} $(\x_*, \y_*)$ that satisfies
%\begin{align*}
$f(\x_*, \y)\leq f(\x_*, \y_*)\leq f(\x, \y_*), \forall \x\in\X, \forall\y\in\Y$.
%\end{align*}
However, finding such a saddle point for \eqref{eqn:P} is in general NP-hard (in fact, minimizing a generic non-convex function is already NP-hard (see e.g.,~\citep{Hillar:2013:MTP:2555516.2512329}). Following the existing studies on non-convex optimization, we may instead consider finding a first-order \emph{stationary} solution to~(\ref{eqn:P}), which is a solution $(\x, \y)\in\mathcal F_*$ with
\begin{align}
\label{eq:Fstar}
\mathcal F_*=\left\{(\x, \y)\in \X\times \Y\bigg\vert
\begin{array}{l}
0\in\partial_x [f(\x, \y)+ 1_{\X}(\x)],\\ 
0\in\partial_y [-f(\x, \y) + 1_{\Y}(\y)]
\end{array}
\right\},
\end{align}
where $\partial_x$ and $\partial_y$ are partial subdifferentials defined in Section~\ref{sec:pre}, and $1_{\X}(\x)$ and $1_{\Y}(\y)$ are the indicator functions of the sets $\X$ and $\Y$, respectively. It is notable that $(\x, \y)\in \F_*$ is the first-order necessary condition for $(\x, \y)$ to be a (local) saddle point of~(\ref{eqn:P}). However, iterative algorithms typically do not guarantee an exact stationary solution within finitely many iterations. Hence, to establish  non-asymptotic convergence of algorithms for \eqref{eqn:P}, we focus on finding a \emph{nearly $\epsilon$-stationary} solution  which is a useful notion of stationarity when a problem is non-smooth or constrained. 

\begin{table*}[t]
	\caption{Summary of complexity results for finding a nearly $\epsilon$-stationary solution of the $\rho$-weakly-convex-weakly-concave min-max saddle-point problems (where $\rho>0$) by using different algorithms to solve the strongly monotone subproblems. In the stochastic and deterministic settings, the complexity refers to the iteration complexity and, in the finite-sum setting, it refers to gradient complexity (i.e., the number of gradients computed). $L$ refers to the Lipschitz constant. Note that the per-iteration cost in deterministic setting is the same as evaluating the (sub)gradient, while in the stochastic and finite-sum setting it is the same as evaluating the stochastic gradient based on one individual component function.% and $\widetilde O()$ suppresses a logarithmic factor. 
		%The dependence on $\rho, L$ is only highlighted under the condition of Lipschitz continuity in order to compare different methods.   
	}
	%\vspace*{0.1in}
	\centering
	\label{tab:2}
	\scalebox{1}{\begin{tabular}{l|l|l|ll}
			\toprule
			%&This paper&\cite{besbes-2013-optimization}&\cite{DBLP:conf/aistats/JadbabaieRSS15}\\
			% \midrule
			%\multicolumn{2}{r|}{Definition of variation} &path variation &functional variation \\
			%\bottomrule
			%\toprule
			Setting&  Algorithms for sub-problems &Lipschitz  &Complexity\\
			\midrule
			%&& expectation or high probability&\\
			Stochastic&Stochastic Subgradient Method&No&$ O\left(\max(\rho^6, \rho^3)/\epsilon^6\right)$\\
			\midrule
			&Subgradient Method&  No & {$ O\left(\max(\rho^6, \rho^3)/\epsilon^6\right)$}\\ 
			Deterministic &Gradient Descent Method&Yes&$\widetilde O\left(L^2/\epsilon^2\right)$\\ 
			&Extragradient Method&Yes&$\widetilde O\left( L\rho/\epsilon^2 \right)$\\ 
			&Nesterov's Accelerated Method&Yes&$\widetilde O\left( L\rho/\epsilon^2 \right)$\\ 
			\midrule
			%	&Stochastic Gradient Method&Yes&$ O\left(n/\epsilon^2 + 1/\epsilon^4 \right)$\\ 
			Finite-sum ($n$ summands) &Variance Reduction &Yes&$\widetilde O\left(n\rho^2/\epsilon^2 + L^2/\epsilon^2\right)$\\
			%(with $n$ components)   &Accelerated Variance Reduction&Yes&$\widetilde O\left(n\rho^2/\epsilon^2 + L\rho/\epsilon^2\right)$\\
			\bottomrule
	\end{tabular}}
	%\vspace{-0.2in}
\end{table*}

Although $\F_*$ is a direct extension of the set of first-order stationary points of minimization problems to min-max problems, it remains unclear how to prove an algorithm converges non-asymptotically to $\F_*$ when the min-max problem is non-convex non-concave. To the best of our knowledge, {\bf this is the first work that proves the non-asymptotic convergence} of first-order methods to a nearly stationary solution of a class of non-smooth non-convex non-concave min-max problems.  The key novelty of our analysis is viewing min-max problems through the lens of variational inequalities. 

	Let $F(\bz):\mathcal{Z}\rightrightarrows \mathbb{R}^d$ be a set-valued mapping and $\mathcal{Z}\subset \mathbb{R}^d$ be a closed convex set. The \emph{variational inequality} (VI) problem, also known as the \emph{Stampacchia variational inequality} (SVI) problem~\citep{hartman1966some}, associated with $F$ and $\mathcal{Z}$, is denoted by $\text{SVI}(F,\mathcal{Z})$ and concerns finding $\bz^*\in\mathcal{Z}$ such that 
	
	\begin{eqnarray}\label{eq:SVI}
	\exists \bxi^*\in F(\bz^*)\text{ s.t. }\left\langle \bxi^*, \bz-\bz^*\right\rangle\geq 0,~\forall  \bz\in\mathcal{Z}.
	\end{eqnarray}
	A closely related but different problem is the \emph{Minty variational inequality} (MVI) problem~\citep{minty1962monotone} associated with $F$ and $\mathcal{Z}$, which is denoted by $\text{MVI}(F,\mathcal{Z})$ and concerns finding $\bz_*\in\mathcal{Z}$ such that 
	\begin{eqnarray}\label{eq:MVI}
	\left\langle \bxi, \bz-\bz_*\right\rangle\geq 0,~\forall  \bz\in\mathcal{Z},\forall  \bxi\in F(\bz).
	\end{eqnarray}
	The SVI and MVI corresponding to problem~(\ref{eqn:P}) are defined with the set $\Z=\X\times\Y$ and
	the mapping $F(\bz) := (\partial_x f(\x, \y),  \partial_y[- f(\x, \y)])^{\top}$ where $\bz=(\x,\y)$.  
	The main contributions of this paper are summarized below:
	\begin{itemize}[leftmargin=*]
		%\vspace*{-0.1in}
		\item We propose a generic algorithm motivated by the inexact proximal point method~\citep{davis2017proximally} for solving a class of non-convex non-concave min-max problems, whose objective function $f(\x, \y)$ is weakly convex in $\x$ for any fixed $\y\in\Y$ and weakly concave in $\y$  for any fixed $\x\in\X$. The algorithm consists of approximately solving a sequence of strongly monotone VIs constructed by adding a strongly monotone mapping to $F(\bz)$ with a sequentially updated proximal center.

		\item We prove the theoretical convergence of the proposed algorithm under the key assumption that \emph{there exists a solution to the $\text{MVI}(F,\mathcal{Z})$ corresponding to~(\ref{eqn:P})}. We establish the iteration complexities for finding a nearly $\epsilon$-stationary solution when different algorithms, including the stochastic subgradient method, the gradient descent method, extragradient method, and the variance reduction methods, are employed as a subroutine for solving each strongly monotone VI in the proposed framework. In particular, the iteration complexity is $O(1/\epsilon^6)$ when using the stochastic subgradient method and is $\widetilde O(1/\epsilon^2)$ when using the gradient method, or the extragradient method under the additional smoothness assumption of $f$.~\footnote{$\widetilde O$ suppresses all logarithmic factors.} Here, the complexity when using Nesterov's accelerated method or the extragradient method improves the one when using the gradient method by a constant factor when a condition number of the original problem is large. Furthermore, if $f$ is of a finite-sum form we can use variance reduction methods to improve the complexity.  The achieved complexity results are presented in Table~\ref{tab:2}.
		%\vspace*{-0.1in}
		
		\item Moreover, our algorithms are directly applicable  to the more general problem of solving the SVI when $F$ is weakly monotone, and our analysis directly implies the non-asymptotic convergence to a nearly $\epsilon$-accurate solution to the SVI under the condition that a solution to the corresponding MVI exists. 
	\end{itemize}
	
%	\vspace*{-0.1in}

	{\bf Application in Training GAN.}\label{sec:apps}
	The formulation \eqref{eqn:P} has broad applications in machine learning, statistics and operations research. Here we present one example in machine learning for training GAN~\citep{Goodfellow:2014:GAN:2969033.2969125,pmlr-v70-arjovsky17a,DBLP:conf/nips/GulrajaniAADC17}.  GAN refers to a powerful class of generative models that cast generative modeling as a game between two networks: a generator network produces synthetic data out of noise and a discriminator network discriminates between the true data and the generator's output. Let us consider WGAN~\citep{pmlr-v70-arjovsky17a}, a recently proposed variant of GAN, as an example. The min-max problem corresponding to WGAN can be written as 
	\begin{align*}
	\min_{\theta\in\Theta}\max_{\w\in\mathcal W} \left\{\E_{\x\sim\mathbb P_r}[f_\w(\x)] - \E_{\z\sim\mathbb P_\z}[f_\w(g_\theta(\z))]\right\},
	\end{align*}
	where $f_\w(\x)$ denotes a Lipschitz continuous function parameterized  by $\w$ corresponding to the discriminator, $g_\theta(\z)$ denotes the parameterized function corresponding to the generator, $\mathbb P_r$ denotes the underlying distribution of the data $\bx$, and $\mathbb P_\z$ denotes the distribution of noise $\bz$. Functions $f_\w(\cdot)$ and $g_\theta(\cdot)$ are usually represented by deep neural networks. When the deep neural networks induce smooth $f_\w(\cdot)$ and $g_\theta(\cdot)$ (for example by using smooth activation functions), it can be showed that the objective function of WGAN is weakly convex in $\theta$ and weakly concave in $\w$.   
	More applications of~(\ref{eqn:P}) with weakly-convex-weakly-concave function $f(\x, \y)$ can be found in reinforcement learning~\citep{DBLP:conf/icml/DaiS0XHLCS18}, learning a robust model under heavy-tailed noise~\citep{audibert2011},   adversarial learning~\citep{DBLP:journals/corr/abs-1710-10571}, etc (cf. examples provided in the Appendix). 
	
\section{Related Work}
%\vspace*{-0.1in}%In this section, we review some work related to solving the min-max problem~(\ref{eqn:P}) or the SVI problem~(\ref{eq:SVI}).

There is a growing interest on  first-order algorithms for solving non-convex problems, e.g., \citep{DBLP:journals/siamjo/GhadimiL13a,yangnonconvexmo,DBLP:journals/mp/GhadimiL16,DBLP:conf/cdc/ReddiSPS16,Reddi:2016:SVR:3045390.3045425,DBLP:journals/corr/abs/1805.05411,DBLP:conf/icml/Allen-Zhu17,DBLP:conf/icml/ZhuH16,sgdweakly18,davis2018stochastic,zhang2018convergence,davis2018stochastic}, and this list is by no means complete. %However, most of them contribute to the non-convex minimization problem instead of a min-max problem. 
The proposed algorithm shares the similarity with several of these previous works~\citep{davis2017proximally,chen18stagewise,DBLP:journals/corr/abs/1805.05411,DBLP:conf/icml/Allen-Zhu17,chen18stagewisekatyusha} by using techniques related to the proximal-point method, which %The idea of proximal-point method is to solve a sequence of proximal subproblems exactly or inexactly that are formed by  adding  a strongly convex term to the original objective function with sequentially changed proximal centers and/or regularization parameters. It 
has a long history~\citep{citeulike:9472207}. %Recently, this idea has been adopted for solving non-convex minimization when the objective function is weakly convex, meaning that the function becomes convex after adding a strongly convex quadratic term. In that case, the subproblems in the proximal-point method are convex and can be solved efficiently with an appropriate algorithm. For example, \citet{DBLP:journals/siamjo/CarmonDHS18} use accelerated gradient method for solving the subproblems (under smoothness assumption) and \citet{davis2017proximally} employed  a stochastic subgradient method.  \citet{chen18stagewise} analyzed a framework that can use a broad class of stochastic algorithms (e.g., adaptive gradient methods) for solving the subproblems. \citet{DBLP:journals/corr/abs/1805.05411,DBLP:conf/icml/Allen-Zhu17,chen18stagewisekatyusha} have also utilized the idea similar to the proximal-point method for solving finite-sum smooth minimization problems with improved complexity. 
%More recently, \citet{hassan18nonconvexmm} have considered using the proximal-point method to solve weakly-convex and concave min-max problems. 
The work of~\citep{shamir2020can,zhang2020complexity} considered first-order algorithms for nonsmooth nonconvex minimization problems.
However, their analysis cannot be directly extended to weakly-convex weakly-concave min-max problems, which is more challenging. 
%In particular, all existing works whose analysis are built on the objective gap convergence for the proximal subproblems are not applicable to our considered min-max problem. The key novelty of our analysis lies on using the tools of variational inequalities, thus avoiding using the objective gap or the duality gap of the subproblems. 

%\paragraph{Solving Non-Convex  Min-Max Problems in Machine Learning.}
Several recent works~\citep{NIPS2017_7056,DBLP:journals/corr/abs-1805-07588,DBLP:journals/corr/abs-1710-10571,hassan18nonconvexmm,DBLP:conf/icml/DaiS0XHLCS18} have considered non-convex min-max problems and their applications in machine learning. However, their  algorithms and analysis are either built on restricted assumptions (e.g., the maximization problem can be solved exactly) or only applied to a much smaller family non-convex min-max problems. %For example, \citet{NIPS2017_7056,DBLP:journals/corr/abs-1710-10571,DBLP:conf/icml/DaiS0XHLCS18} assume the maximization problem in \eqref{eqn:P} can be solved exactly, which may be impossible or expensive. %\citet{DBLP:journals/corr/abs-1805-07588} solve a min-max robust learning problem with a primal-dual stochastic gradient algorithm but they require the maximization problem to be concave. 
\citet{hassan18nonconvexmm} considered weakly-convex and concave min-max problems. % that cover many applications in machine learning (e.g., distributionally robust learning).  
Relying on the concavity of the maximization part, they are able to establish the convergence to a nearly stationary point of the function $\max_{\y\in\Y}  f(\x, \y)$. \citet{thekumparampil2019efficient} developed deterministic first-order algorithms for strongly-convex-concave and nonconvex-concave min-max problems respectively. ~\citet{lin2020near} provided near-optimal algorithms for strongly-convex-strongly-concave min-max problems. In contrast, the problem considered in this paper is weakly-convex and weakly-concave and covers even more applications in machine learning (e.g., GAN training).  
Recently, there emerges a wave of studies that analyze the convergence properties of some algorithms for training GAN~\citep{doi:10.1137/15M1026924,DBLP:conf/nips/anageas18,DBLP:journals/corr/abs-1711-00141,DBLP:conf/nips/HeuselRUNH17,DBLP:conf/nips/NagarajanK17,DBLP:journals/corr/GrnarovaLLHK17}. %But they do not provide any complexity results and their analysis are mostly asymptotic.  %analyzed the asymptotic convergence of the gradient dynamics for a min-max problem. Their analysis focus on the conditions that guarantee the asymptotic convergence to the saddle-points of a min-max problem. In addition, their results require the gradients to be locally Lipschitz continuous. In contrast, our analysis cover the case where gradients are not Lipschitz continuous and our convergence is non-asymptotic. \citet{DBLP:conf/nips/anageas18} analyzed the limiting behavior of two gradient-based dynamics, namely gradient dynamics and optimistic gradient dynamics for an unconstrained min-max problem with a twice continuously differentiable objective function and Lipschitz continuous gradients. Their analysis focus on the stability of the limiting points when the dynamics converge. However, convergence might not be guaranteed under the two dynamics they considered. 
%Recently, several papers analyzed the convergence properties of some algorithms for training GAN~\cite{DBLP:journals/corr/abs-1711-00141,DBLP:conf/nips/HeuselRUNH17,DBLP:conf/nips/NagarajanK17,DBLP:journals/corr/GrnarovaLLHK17}. 
However, their results are either asymptotic~\citep{DBLP:journals/corr/abs-1711-00141,DBLP:conf/nips/HeuselRUNH17,DBLP:conf/nips/NagarajanK17,doi:10.1137/15M1026924} or their analysis require strong assumptions of the problem~\citep{DBLP:conf/nips/NagarajanK17,DBLP:journals/corr/GrnarovaLLHK17} (e.g., the problem is concave in maximization). %In addition, almost all  the results mentioned above except~\citet{DBLP:journals/corr/GrnarovaLLHK17} considered the unconstrained case. 
We also notice that two recent papers \citep{sim18ganvi,panomd18} have also considered the algorithms for min-max problems from the perspective of variational inequalities. However, the analysis in~\citep{sim18ganvi} is for convex-concave problems or monotone variational inequalities. \citet{panomd18} proved asymptotic convergence to saddle points of the min-max problem under a strong coherence assumption, i.e., every saddle point is a solution to the corresponding MVI. In contrast, we only assume the MVI has a solution and prove non-asymptotic convergence to a nearly $\epsilon$-stationary point.% of the min-max problem. 

%analyzed the convergence/divergence of mirror descent and optimistic mirror descent methods for solving a non-convex non-concave saddle-point problem. Assuming coherence between the min-max problem and its corresponding MVI problem (i.e., every saddle point of the  min-max problem is a solution to the MVI),  they proved asymptotic convergence to saddle points of the min-max problem. In contrast, we only assume about the MVI has a solution and proves non-asymptotic convergence to a nearly $\epsilon$-stationary point of the min-max problem. 

%\paragraph{Solving SVI.}
Both SVI and MVI have a long history in the literature~\citep{harker1990finite}. %When the set-valued mapping $F$ is monotone, many efficient algorithms with non-asymptotic convergence guarantee have been developed for a VI problem itself or under the setting of a min-max problem ~\cite{10017556617,rockafellar1976monotone,juditsky2011solving,nesterov2007dual,bruck1977weak,nemirovski-2005-prox,nemirovski2009robust,nedic2009subgradient,burachik2010inexact}.
When the set-valued mapping $F$ is non-monotone, there have been many studies that design and analyze algorithms for finding a solution to the SVI problem~\citep{Solodov1999ANP,bao05pr,DBLP:journals/corr/corr1609,allevi2006proximal,DBLP:journals/coap/DangL15,DBLP:journals/siamjo/IusemJOT17}. However, the main difference between these works and the present work is that their convergence analysis is asymptotic except for~\citep{DBLP:journals/coap/DangL15,DBLP:journals/siamjo/IusemJOT17}. %The connection of these works to ours is that they also  assume that a solution exists to the associated MVI, which is true for certain problems (e.g., when the set-valued mapping $F(\z)$ is pseudomonotone~\cite{aussel2004quasimonotone,crouzeix1997pseudomonotone,daniilidis1999coercivity,daniilidis1999characterization}). %To our knowledge, 
\citet{DBLP:journals/coap/DangL15}  established the first non-asymptotic convergence for non-monotone SVI by deterministic algorithms when the mapping $F(\bz)$ is single-valued and Lipschitz (or H\"{o}lder) continuous. \citet{DBLP:journals/siamjo/IusemJOT17} analyzed a variant of stochastic extragradient method with increasing mini-batch size for solving SVI with a single-valued and Lipschitz continuous mapping and obtained several complexity results but for a different convergence measure from~\citep{DBLP:journals/coap/DangL15} and ours.  %However, for  non-Lipschitz and non-H\"{o}lder continuous mappings, their convergence result is only asymptotic. 
In contrast, our work provides the first non-asymptotic convergence of stochastic and deterministic algorithms for solving the SVI problems with a set-valued mapping that is non-Lipschitz and non-H\"{o}lder continuous but weakly monotone. %Moreover, when  $F(\bz)$ is single-valued and Lipschitz continuous, we show that in Section~\ref{sec:last} the algorithm in~\cite{DBLP:journals/coap/DangL15} may have a higher complexity than ours when used to find nearly stationary points of min-max problems. %the presented framework also brings us several new interesting algorithms with  improved complexity than that established in~\cite{DBLP:journals/coap/DangL15} for solving the min-max saddle-point problems. 
%For more detailed comparison, please refer to the discussion after Corollaries~\ref{thm:gd} and~\ref{thm:eg2} for solving the SVI problem and the min-max saddle-point problem, respectively.
%Probabilistic inference has become a core technology in AI,
%largely due to developments in graph-theoretic methods for the 
%representation and manipulation of complex probability 
%distributions~\citep{pearl:88}.  Whether in their guise as 
%directed graphs (Bayesian networks) or as undirected graphs (Markov 
%random fields), \emph{probabilistic graphical models} have a number 
%of virtues as representations of uncertainty and as inference engines.  
%Graphical models allow a separation between qualitative, structural
%aspects of uncertain knowledge and the quantitative, parametric aspects 
%of uncertainty...\\
%
%%{\noindent \em Remainder omitted in this sample. See http://www.jmlr.org/papers/ for full paper.}
%%
%%% Acknowledgements should go at the end, before appendices and references
%%
%%\acks{We would like to acknowledge support for this project
%%from the National Science Foundation (NSF grant IIS-9988642)
%%and the Multidisciplinary Research Program of the Department
%%of Defense (MURI N00014-00-1-0637). }
%
%% Manual newpage inserted to improve layout of sample file - not
%% needed in general before appendices/bibliography.
%
%\newpage

\section{Preliminaries}\label{sec:pre}
%\vspace*{-0.1in}
%In this paper, we present stochastic first-order methods to find nearly stationary points of the following non-convex optimization problem
%\begin{equation}\label{originalproblem}
%\min_{\bx\in \mathcal{X}} \bigg\{\psi(\bx):=\max_{\by\in \Y} \sum^n_{i=1} y_i f_i(\bx) - r(\by) \bigg\}
%\end{equation}
%where $\mathcal{X} \subset \reals^d$ is a convex compact set, $\Y: = \{\by \in \reals^n~|~ \sum^n_{i=1} y_i = 1,~ y_i \geq 0~ i=1,....,n\}$, $r: \Y \rightarrow \reals$ is closed and $\mu_y$-strongly convex ($\mu_y>0$), $f_i: \reals^d \rightarrow \reals$ is $\rho$-weakly function, i.e., $f_i(\bx)+\frac{\rho}{2}\|\bx\|_2^2$ is convex for $i=1,2,\dots,n$.   
We present some preliminaries in this section. For simplicity, we consider \eqref{eqn:P} defined in the Euclidean space with inner product $\langle \bz,\bz'\rangle=\bz^\top\bz'$. We use $\|\cdot\|$ to represent the Euclidean norm.  Let $\text{Proj}_{\Z}[\z]$ denote an Euclidean projection mapping that projects $\z$ onto the set $\Z$. 
Given a function $h:\reals^d\rightarrow \reals\cup\{+\infty\}$, we define the  \emph{(Fr\'echet) subdifferential} of $h$ as 
%\small
\begin{align}
\label{subgradient}
&\partial h(\bx)=\left\{\bzt\in\mathbb{R}^d\bigg| 
\begin{array}{l}
h(\bx')
\geq h(\bx)+\bzt^\top(\bx'-\bx)+o(\|\bx'-\bx\|),
\\ \bx'\rightarrow\bx
\end{array}
\right\},
\end{align}
%\normalsize
where each element in $\partial h(\bx)$ is called a \emph{(Fr\'echet) subgradient} of $h$ at $\bx$. In this paper, we will analyze the convergence of an iterative algorithm for solving~(\ref{eqn:P}) through the lens of variational inequalities. To this end, we first introduce some background related to variational inequalities. %Let $\omega:\Z\rightarrow\reals$ be a $1$-strongly convex function with respect to $\|\cdot\|$ which is differentiable in $\text{int}\Z$. The Bregman divergence associated to $\omega$ is defined as
%$V:\Z\times\text{int}\Z\rightarrow \mathbb{R}$ satisfying 
%\begin{eqnarray}
%\label{eq:breg}
%V(\bz,\bz')&:=&\omega(\bz)-\omega(\bz')-\left\langle\nabla  \omega(\bz'),\bz-\bz'\right\rangle.
%\end{eqnarray}
%Note that we have $V(\bz,\bz')\geq \frac{1}{2}\|\bz-\bz'\|^2$.
%\subsection{Variational Inequalities }
%The following notions are classical:
\begin{definition}
	\label{def:monotone}
	A set-valued mapping $F(\bz):\Z\rightrightarrows \mathbb{R}^d$  is said to be {monotone}
	if
	%\begin{eqnarray}\label{eq:monotone}
	$\left\langle \bxi-\bxi', \bz-\bz'\right\rangle\geq 0$, 
	%\end{eqnarray}
	{$\mu$-strongly monotone} for $\mu>0$ if 
	%\begin{eqnarray}\label{eq:stmonotone}
	$\left\langle \bxi-\bxi', \bz-\bz'\right\rangle\geq \mu\|\bz-\bz'\|^2$, 
	%\end{eqnarray}
	and {$\rho$-weakly monotone} for $\rho>0$ if 
	%\begin{eqnarray}\label{eq:wkmonotone}
	$\left\langle \bxi-\bxi', \bz-\bz'\right\rangle\geq -\rho\|\bz-\bz'\|^2$,
	%\end{eqnarray}
	for any $\bz,\bz'\in\mathcal{Z}$, $\bxi\in F(\bz)$, and $\bxi'\in F(\bz')$.
\end{definition}
%\vspace*{-0.05in}
By a slight abuse of notation, when $F(\bz)$ is a singleton set, we will use $F(\z)$ to represent the single element in the set. 
%\begin{definition}
A (single-valued) mapping $F(\bz):\Z\rightarrow \mathbb{R}^d$  is said to be $L$-Lipschitz continuous if
%\begin{eqnarray}\label{eq:lip}
$\|F(\z)- F(\z')\|\leq L\|\z - \z'\|,~\forall  \bz,\bz'\in\mathcal{Z}$. It is well-known that $\text{SVI}(F,\Z)$ has a \emph{unique solution} if $F$ is $\mu$-strongly monotone~\citep{1078-0947Nesterov}.  
%\end{eqnarray}
%\end{definition}

For the min-max problem~(\ref{eqn:P}), we define $\z=(\bx,\by)^{\top}$, $F(\bz)\equiv (\partial_x f(\x, \y), - \partial_y f(\x, \y))^{\top}$ and $\Z\equiv\X\times \Y$. The SVI corresponding to~(\ref{eqn:P}) is defined by such a $F(\bz)$ and $\Z$. 
If $F(\bz)$ is $L$-Lipchitz continuous it is also $L$-weakly monotone. 
%By Cauchy-Schwarz inequality, $F(\z)$ is $L$-weakly convex if it is single-valued and $L$-Lipschitz continuous.
However, Lipchitz continuity of $F$ is not necessary for $F$ to be weakly monotone. Below, we show that if $f(\x, \y)$ is weakly convex in terms of $\x$ and weakly concave in terms of $\y$, then its corresponding $F$ is weakly monotone. To this end, we first introduce the definition of weakly convex and weakly concave. 

\begin{definition} $f(\x,\y)$ is $\rho$-weakly-convex-weakly-concave if for any $\y\in\Y$, $f(\x, \y)+ \frac{\rho}{2}\|\x\|^2$ is convex in  $\x$, and,  for any $\x\in\X$, $f(\x, \y) -  \frac{\rho}{2}\|\y\|^2$ is concave  in $\y$. 
	%		\label{assume:MVIexist} 
	%		\begin{itemize}
	%			\item The set $\Z$ is compact, i.e., there exists $D>0$  such that $\max_{\bz,\bz'\in\Z}\|\bz-\bz'\|\leq D$. 
	%			\item The mapping $F$ is $\rho$-weakly monotone.
	%			\item The $\text{MVI}(F)$ problem has a solution. 
	%		\end{itemize}
\end{definition} 
%\vspace*{-0.08in}
In Section~\ref{sec:examples} in Appendix, we present some examples of the min-max problem whose objective function is weakly-convex and weakly-concave, which are not necessarily smooth functions.  %Since we will use the results presented in last section to study the convergence for solving a min-max problem, 
The following lemma shows the connection between weak-convexity weak-concavity and weak-monotonicity. Its proof and all missing proofs are included in the~Appendix.%justifies the Assumption~\ref{assume:MVIexist} (ii) when $f(\x,\y)$ is $\rho$-weakly-convex and $\rho$-weakly-concave. 
\begin{lemma}~\label{lem:weaklymonotone}
	$f(\bx,\by)$ is $\rho$-weakly-convex-weakly-concave if and only if $F(\bz)$ is $\rho$-weakly monotone. 
\end{lemma}
%\vspace*{-0.08in}
The following lemma is standard but critical for obtaining our results. The counterpart of this lemma for minimization problems can be found in~\citep{davis2017proximally}.
\begin{lemma}\label{lem:PPM}
	If  $F(\bz):\Z\rightrightarrows \mathbb{R}^d$ is $\rho$-weakly monotone, the mapping $F_{\bw}^\gamma(\bz)\equiv F(\bz)+\frac{1}{\gamma}(\bz-\bw)$ is %\footnote{Given a set $A\subset\mathbb{R}^d$, we define $A+\bz=\{\bw+\bz|\bw\in A\}$.} 
	 $(\frac{1}{\gamma}-\rho)$-strongly monotone on $\Z$ for any $0<\gamma<\rho^{-1}$ and any $\bw\in\Z$.
\end{lemma}
%\vspace*{-0.08in}
 A point $\z\in\Z$ is called first-order stationary point of~(\ref{eqn:P}) if 
$\z\in \mathcal F_*$ with $\mathcal F_*$ defined in \eqref{eq:Fstar}.
An iterative algorithm typically only finds an $\epsilon$-stationary solution which is a solution $\z\in \Z$ such that
%\begin{align}
%\label{eq:minmaxstationary}
%\text{dist}^2(\mathbf{0}, \partial_x [f(\x, \y)+ 1_{\X}(\x)]\times \partial_y [-f(\x, \y) + 1_{\Y}(\y)])\\
$\text{dist}(\mathbf{0}, \partial (f(\x, \y)+ 1_{\Z}(\x,\y))\leq \epsilon$, 
%\end{align}
where $\partial (f(\x, \y)+ 1_{\Z}(\x,\y))\equiv \partial_x [f(\x, \y)+ 1_{\X}(\x)]\times \partial_y [-f(\x, \y) + 1_{\Y}(\y)]$ and 
$\text{dist}(\z, \mathcal{S})$ denotes the Euclidean distance from a point to a set $\mathcal{S}$. 
The non-smoothness nature of the problem makes it challenging to find an  $\epsilon$-stationary solution for an iterative algorithm.  For example, consider the convex minimization $\min_{z\in[-1,1]}|z|$ and has a solution at $0$. Hence,  if $\bar z$ is very close to $0$ but not $0$, we always have $|\bar\xi|=1$ for any $\bar\xi\in\partial|\bar z|$.  Hence,  we consider the following notion of a nearly stationary point for a non-smooth min-max problem. 
\begin{definition}
	A point $\w\in\Z$ is called a nearly $\epsilon$-stationary solution to~(\ref{eqn:P}) if there exists $\bar{\w}=(\bar{\mathbf{u}}, \bar{\mathbf{v}})^{\top}\in\Z$ and $c>0$ such that 
%\begin{align*}
$\|\w -\bar{\w}\|\leq c\epsilon$, and $\text{dist}(0, \partial (f(\bar{\mathbf{u}}, \bar{\mathbf{v}})+ 1_{\Z}(\bar{\mathbf{u}}, \bar{\mathbf{v}}))]\leq \epsilon$.
%\end{align*}
\end{definition}
Such a notion  of nearly stationary  has been utilized in several works for tackling non-smooth non-convex minimization problems~\citep{sgdweakly18,davis2017proximally,chen18stagewise}. %To make this problem tractable,  the following assumption is made regarding the min-max saddle-point problem~(\ref{eqn:P}).
%Although we can use the convergence measure for the SVI problem defined in previous sections to measure the convergence of an iterative algorithm for solving (\ref{eqn:P}), we are interested in a more direct measure regarding the stationarity of a solution for~(\ref{eqn:P}).  In order to show the existence of nearly $(\epsilon,\delta)$-gap solutions, we define a \emph{proximal-point mapping} (PPM) of $F$ as 
In order to show the existence of a nearly $\epsilon$-stationary solution, we define a \emph{proximal-point mapping} of $F$ at a proximal center $\w$ as 
\begin{eqnarray}
\label{eq:ppm}
F_{\bw}^\gamma(\bz)\equiv F(\bz)+\frac{1}{\gamma}(\bz-\bw)
\end{eqnarray} 
where $0<\gamma<\rho^{-1}$. According to Lemma~\ref{lem:PPM}, $F_{\bw}^\gamma$ is $(\frac{1}{\gamma}-\rho)$-strongly monotone so that $\text{SVI}(F_{\bw}^\gamma,\Z)$ has a unique solution denoted by $\bar\bw$.  	 The following lemma is a foundation of our algorithm and analysis, which shows that as long as we find a solution $\w$ such that $\|\w - \bar\w\|\leq \gamma\epsilon$, it is a nearly $\epsilon$-stationary solution. %will be used to prove that an obtained solution is nearly $\epsilon$-stationary. 
\begin{lemma}\label{lem:zbw2}
	Let  $F_{\bw}^\gamma$ be defined in \eqref{eq:ppm} for $0<\gamma<\rho^{-1}$ and $\bw=(\mathbf{u}, \mathbf{v})^{\top}\in\Z$. Denote by $\bar{\bw}=(\bar{\mathbf{u}}, \bar{\mathbf{v}})^{\top}$ the solution to $\text{SVI}(F_{\bw}^\gamma,\Z)$.  We have 
	\begin{eqnarray}
	\label{eq:zbw2}
	\text{dist}(0, \partial (f(\bar{\mathbf{u}}, \bar{\mathbf{v}})+ 1_{\Z}(\bar{\mathbf{u}}, \bar{\mathbf{v}}))\leq \|\bw-\bar{\bw}\|/\gamma.% ,~\forall  \bz\in\mathcal{Z}.
	\end{eqnarray} 
\end{lemma}
%\vspace*{-0.1in}
Before ending this section, we formally present the basic assumptions used in our analysis. 
%In this paper, we do not assume the set-valued mapping $F$ to be monotone and aim to solve the SVI problem. In order to make the  SVI problem tractable,  the following assumption is made throughout the paper and is critical to establish all results in this paper.
\begin{assumption}~\label{ass:3}%\vspace*{-0.1in}
	%\label{assume:MVIexist} 
	%\begin{itemize}
	(i) The set $\Z$ is convex and compact so that there exists $D>0$ such that $\max_{\bz,\bz'\in\Z}\|\bz-\bz'\|\leq D$. 
	(ii) The mapping $F$ is $\rho$-weakly monotone.
	(iii) The $\text{MVI}(F,\Z)$ problem has a solution. 
	%\end{itemize}
\end{assumption}
%\vspace*{-0.1in}
It is notable that the last assumption has been used by most previous works for solving non-monotone SVI~\citep{Solodov1999ANP,Wang2001,bao05pr,DBLP:journals/corr/corr1609,DBLP:journals/coap/DangL15,DBLP:journals/siamjo/IusemJOT17}. It can be shown that when $F(\z)$ satisfies some generalized notion of monotonicity (e.g., pseudomonotone, quasi-monotone),  a solution of $\text{MVI}(F ,\Z)$  exists. Indeed,  a similar assumption (for non-convex minimization) has been made  for analyzing the convergence of stochastic gradient descent for learning neural networks~\citep{NIPS2017_6662} and also was observed in practice~\citep{pmlr-v80-kleinberg18a}.

%\vspace*{-0.1in}
\section{A Generic Algorithmic Framework with a General Convergence Result}\label{sec:5}
%\vspace*{-0.1in}
In this section, we will present  a generic algorithmic framework for solving the  saddle-point problem~\eqref{eqn:P} through the lens of VI.  The method we propose is called the inexact proximal point (IPP) method. This method consists of solving a sequence of strongly monotone SVIs defined by the mappings $F_{\z_k}^\gamma$ in \eqref{eq:ppm} with a sequentially updated proximal center $\z_k$. In particular, an appropriate first-order algorithm is employed to find an approximate solution $\z_{k+1}$ to  $\text{SVI}(F_{\z_k}^\gamma,\Z)$, which is then used as the proximal center that defines $\text{SVI}(F_{\z_{k+1}}^\gamma,\Z)$. This method is described in Algorithm~\ref{alg:meta}. A subroutine $\text{ApproxSVI}(F_k, \Z, \z_k, \eta_k, T_k)$ is called to approximately solve  $\text{SVI}(F_k,\Z)$, where $\z_k$ is the initial solution of the subroutine,  $\eta_k$ denotes the step size and $T_k$ denotes the number of iterations performed in the subroutine. 
\setlength{\textfloatsep}{5pt}% Remove \textfloatsep
\begin{algorithm}[t]
	\caption{Inexact Proximal Point  (IPP)  Method for Weakly-Monotone SVI}\label{alg:meta}
	\begin{algorithmic}[1]
		\STATE \textbf{Input:} integer $K\geq 1$, step size $\eta_k>0$, integer $T_k\geq 1$, weight $\theta_k>0$ non-decreasing in $k$, and $0<\gamma<\rho^{-1}$
		\FOR {$k = 0,\ldots, K-1$}
		\STATE Let $F_k\equiv F_{\bz_k}^\gamma=F(\z)+\gamma^{-1}(\z-\bz_k)$
		\STATE $\bz_{k+1}=\text{ApproxSVI}(F_k, \Z, \z_k, \eta_k , T_k)$%\left\{\begin{array}{ll}\text{SG}(F_k,\Z,\bz_k,\eta_k,T_k)&\text{if Assumption~\ref{assume:additional}A holds}\\\text{EG}(F_k,\Z,\bz_k,\eta_k,T_k)&\text{if Assumption~\ref{assume:additional}B holds}\end{array}\right.$
		\ENDFOR
		\STATE \textbf{Output:} $\z_\tau$, where $\tau$ is randomly from $\{0,1,\dots,K-1\}$ with $\text{Prob}(\tau=k) = \frac{\theta_k}{\sum_{l=0}^{K-1}\theta_l}.$%, k=0,1 \dots, K-1$. 
		
	\end{algorithmic}
\end{algorithm}	

The  strength  of our framework  is that it decomposes a complex non-convex non-concave minmax (or non-monotone SVI) problem into a sequence of easier strongly-convex strongly-concave  (or strongly monotone SVI) problems. Hence, it allows one to leverage algorithms and convergence theory for strongly-convex strongly-concave  (or strongly monotone SVI) problems to analyze the convergence for the original problem. Such approach is the first time used for solving non-convex non-concave minmax problems and our analysis is novel.  
A generic convergence result of Algorithm~\ref{alg:meta} conditioned on a particular  sequence of $\z_k$ generated by ApproxSVI is stated below. 

\begin{theorem}\label{thm:meta}
Suppose Assumption~\ref{ass:3} holds, and $\text{ApproxSVI}(F_k, \Z, \z_k, \eta_k , T_k)$ in Algorithm~\ref{alg:meta} ensures      
\begin{align}\label{eqn:sa_converge}
\max_{\z\in\Z}\E[\xi_{k+1}^{\top}(\z_{k+1} - \z)|\z_k]  \leq \frac{c}{k+1}, 
\end{align}
for $k=0,1,\dots,K-1$, where $c>0$, $\xi_{k+1}\in F_k(\z_{k+1})$ and $\E[\cdot|\z_k]$ is the conditional expectation conditioning on $\z_k$.
%all the stochastic events until $\bz_k$ is generated. 
By choosing  $\gamma=\frac{1}{2\rho}$ and $\theta_k =(k+1)^\alpha$ with $\alpha\geq 1$ in Algorithm~\ref{alg:meta}, we have
\begin{align}\label{eqn:fc}%\label{eqn:proximal mapping_converge}
\E[ \|\z_\tau - \bar\z_\tau \|^2]\leq \frac{2D^2(\alpha+1)}{K}+ \frac{4c(\alpha+1)}{K\rho}.%\leq \frac{2D^2\sqrt{K}}{(2/3)K^{3/2}} + \frac{4\sum_{k=0}^{K-1}\sqrt{k+1}\varepsilon(\eta_k, T_k)}{(2/3)\rho K^{3/2}}.
\end{align}
where $\bar\z_\tau$ is the unique solution to $\text{SVI}(F^\gamma_{\z_\tau}, \Z)$. 
\end{theorem}
%\vspace*{-0.08in}
{\bf Remark:} The total complexity of   Algorithm~\ref{alg:meta} for finding a nearly $\epsilon$-stationary solution depends on the complexity of ApproxSVI for computing each $\z_{k+1}$ satisfying~(\ref{eqn:sa_converge}), which corresponds  to solving $\text{SVI}(F_{\z_k}^\gamma,\Z)$ approximately with an increasing accuracy as $k$ increases. %We present the proof of Theorem~\ref{thm:meta} in Section~\ref{sec:proofthmmeta} in Appendix. 

Below, we present several candidates for ApproxSVI with their convergence properties.  We first consider (stochastic) subgradient method without imposing Lipchitz continuity assumption of $F$. In next section, we derive improved rates when $F$ is Lipchitz continuous. With a specific algorithm $\mathcal A$ for solving $\text{SVI}(F_{\z_k}^\gamma,\Z)$ at each stage of Algorithm~\ref{alg:meta},  the resulting algorithm is named as IPP-$\mathcal A$.

Proposition~\ref{prop:alg1} summarizes the convergence result of  stochastic subgradient method for solving SVI$(F_k, \Z)$ approximately, and Corollary~\ref{thm:sgd} is a corollary of Theorem~\ref{thm:sgd} that states the convergence result of IPP-SG.    Here, we only show stochastic subgradient method in Algorithm~\ref{alg:SG}. 

\begin{proposition}\label{prop:alg1}
	%Suppose Assumption~\ref{assume:additional}A holds. 
	Suppose $F_k$ is monotone, $\E[\bzeta(\z)]\in F_k(\z)$, and $\E\|\bzeta(\z)\|^2\leq G_k^2$ for any $\bz\in\Z$.   Algorithm~\ref{alg:SG} applied to SVI$(F_k, \Z)$ guarantees that %for any $\z\in\Z$
	\begin{eqnarray}\label{eqn:ineq1}
	\max_{\z\in\Z}\E[(\bxi^{(\tau)})^{\top}(\z^{(\tau)}-\z)]\leq  \frac{D^2}{2\eta T} + \frac{\eta G_k^2}{2}, \text{	where $\bxi^{(\tau)}\in F_k(\z^{(\tau)})$. }
	\end{eqnarray}
	
	%	If additionally $F$ is $\mu$-strongly monotone, Algorithm~\ref{alg:SG}  guarantees
	%	\[
	%	\mu\E[\|\z^{(\tau)} - \w_*\|^2] \leq \frac{D^2}{2\eta T} + \frac{\eta G^2}{2}.
	%	\]
	%	where $\w_*$ denotes a solution to $\text{SVI}(F, \Z)$.
\end{proposition}
%\vspace*{-0.08in}
{\bf Remark:} %Although (\ref{eqn:ineq1}) is proved without assuming monotonicity of $F$, its does necessarily implies the convergence of $\z^{(\tau)}$ to a solution to $\text{SVI}(F, \Z)$. The reason is that (\ref{eqn:ineq1}) only implies $\max_{\z\in\Z}\E[(\bxi^{(\tau)})^{\top}(\z^{(\tau)}-\z)]\leq \epsilon$ when $T=O(\frac{1}{\epsilon^2})$ and $\eta=O(\epsilon)$, while solving $\text{SVI}(F, \Z)$ requires that $\E[\max_{\z\in\Z}(\bxi^{(\tau)})^{\top}(\z^{(\tau)}-\z)]\leq \epsilon$. The difference comes from that the max and the expectation cannot be switched. On the other hand, 
We would like to emphasize that as long as an algorithm can find a $\z^{(\tau)}$ such that  $\max_{\z\in\Z}\E[(\bxi^{(\tau)})^{\top}(\z^{(\tau)}-\z)]\leq O(1/\sqrt{T})$ in $T$ iterations, it can be used as ApproxSVI and IPP will achieve the same total iteration 
as when the stochastic subgradient method is used. One example that is of particular interest to the deep learning community (e.g., for training GAN~\citep{DBLP:conf/nips/GulrajaniAADC17}) is the Adam-style stochastic algorithm~\citep{kingma2014adam,sashank2018adam}.  

\begin{algorithm}[t]
	\caption{Stochastic Subgradient Method for $\text{SVI}(F, \Z)$: SG$(F,\Z,\bz^{(0)},\eta,T)$}\label{alg:SG}
	\begin{algorithmic}[1]
		\STATE \textbf{Input:} mapping $F$, set $\Z$, $\bz^{(0)}\in\mathcal{Z}$, $\eta > 0$, $T\geq 1$ 
		\FOR{$t=0,..., T-1$}
		\STATE $\z^{(t+1)}= \text{Proj}_{\Z}\left(\z^{(t)} - \eta \bzeta(\z^{(t)})\right)$ where $\bzeta(\z^{(t)})$ satisfies $\E[\bzeta(\z^{(t)})]\in F(\z^{(t)})$.
		\ENDFOR
		\STATE \textbf{Output:} $\z^{(\tau)}$, where $\tau$ is uniformly randomly from $\{0,1,\dots,T-1\}$. 
		
	\end{algorithmic}
\end{algorithm}

\begin{corollary}\label{thm:sgd}
	For problem~(\ref{eqn:P}), assume that Assumption~\ref{ass:3} holds, and for any $\z\in\Z$ there exists $\bzeta(\z)\in F(\z)$ such that  $\E[\bzeta(\z)]\in F(\z)$, and $\E[\|\bzeta(\z)\|^2]\leq G^2$. Suppose Algorithm~\ref{alg:SG} is used as $\text{ApproxSVI}$ with $\eta_k= \frac{D}{G_k(k+1)}$ , $G_k= G + 2\rho D$ and $T_k=(k+1)^2$. By choosing  $\gamma=\frac{1}{2\rho}$, $\theta_k =(k+1)^\alpha$ with $\alpha\geq 1$, and  $K=\frac{(16\rho^2D^2+4\rho DG)(\alpha+1)}{\epsilon^2}$ in Alg.~\ref{alg:meta}, we have
	\begin{align*}%\label{eqn:fc}%\label{eqn:proximal mapping_converge}
	&\E[ \|\z_\tau - \bar{\z}_\tau \|^2]\leq \gamma^2\epsilon^2,\quad \E[\text{dist}^2(0, \partial (f(\bar{\mathbf{u}}_\tau, \bar{\mathbf{v}}_\tau)+ 1_{\Z}(\bar{\mathbf{u}}_\tau, \bar{\mathbf{v}}_\tau))]\leq \epsilon^2%\leq \frac{2D^2\sqrt{K}}{(2/3)K^{3/2}} + \frac{4\sum_{k=0}^{K-1}\sqrt{k+1}\varepsilon(\eta_k, T_k)}{(2/3)\rho K^{3/2}}.
	\end{align*}
	where $\bar\z_\tau=(\bar{\mathbf{u}}_\tau, \bar{\mathbf{v}}_\tau)^{\top}$ is the solution to $\text{SVI}(F^\gamma_{\z_\tau}, \Z)$
	%	 and
	%	\begin{align}\label{eqn:fc}%\label{eqn:proximal mapping_converge}
	%	\exists \bxi\in F(\bar\z_\tau)\text{ s.t. }\E[\max_{\z\in\Z}\langle \bxi, \bar\z_\tau - \bz\rangle]\leq  \epsilon,%\leq \frac{2D^2\sqrt{K}}{(2/3)K^{3/2}} + \frac{4\sum_{k=0}^{K-1}\sqrt{k+1}\varepsilon(\eta_k, T_k)}{(2/3)\rho K^{3/2}}.
	%	\end{align}
	%	
	The  total iteration complexity of $O\left(\frac{\max(\rho^6,\rho^3)}{\epsilon^6}\right)$.
	%$\sum_{k=0}^{K-1}T_k=O(1/\epsilon^6)$. 
\end{corollary}
%\vspace*{-0.08in}
{\bf Remark: } %The above result establishes the convergence result for finding a nearly $(\epsilon,\gamma\epsilon)$-stationary solution. The total iteration complexity can be easily derived as $\sum_{k=1}^Kk^2 = O(1/\epsilon^6)$. 
To the best of our knowledge, this is {\bf the first non-asymptotic convergence} of stochastic subgradient algorithms for solving non-convex non-concave  problems.

%\vspace*{-0.1in}
\section{Improved Rates for Lipchitz Continuous Operator}
%\vspace*{-0.1in}
In this section, we present improved rates when $F$ is single-valued and Lipchitz continuous. In particular, we consider gradient method, extragradient method and variance reduction methods for finite-sum problems. The key to the improved rates is that these methods could achieve faster convergence rates for strongly monotone problems $\text{SVI}(F_k, \Z)$, which have been analyzed in the literature~\citep{1078-0947Nesterov,hassan18nonconvexmm,DBLP:conf/nips/PalaniappanB16,sim18ganvi}. However, there is still a gap between the existing convergence results for strongly monotone problems and the requirement~(\ref{eqn:sa_converge}) in Theorem~\ref{alg:meta} because the existing faster convergence  for strongly monotone problems are for the distance to the optimal solution. The following lemma can bridge the gap that allows us to use existing faster convergence results for Lipchitz continuous operator. 
\begin{lemma}\label{lem:prom}
	Suppose that there exists an algorithm for a monotone $\text{SVI}(F, \Z)$ with $L$-Lipschitz continuous single-valued mapping $ F(\z)$ that returns a solution $\zh$. Let  $\w_*$ be the solution of $\text{SVI}(F, \Z)$.
	Then by constructing $\bar\z = \text{Proj}_{\Z}(\zh - \eta F(\zh))$ with $\eta= 1/(\sqrt{2}L)$, we have
	%\begin{align}\label{eqn:key}
	$\max_{\z\in\Z}F(\bar\z)^{\top}(\bar\z - \z)\leq DL(2+ \sqrt{2})\|\zh - \w_*\|$. 
	%\end{align}
	%where $\w_*$ is the solution of $\text{SVI}(F, \Z)$.
\end{lemma}
%We also consider using other first-other methods to solve $\text{SVI}(F_{\z_k}^\gamma,\Z)$. % once $F$ satisfies the desired properties. 
%In fact, in Section~\ref{sec:otherfom} in Appendix, we present %another two methods, namely an Adam-style stochastic algorithm that are of particular interest to deep learning community (e.g., for training GAN~\cite{DBLP:conf/nips/GulrajaniAADC17}), and 
%the Nesterov's accelerated method~\cite{1078-0947Nesterov} and the extragradient method when $F$  is Lipschitz continuous and single-valued. Both methods improve the complexity of the gradient descent method in the dependence on the condition number defined by $L/\rho$. In that section, we also discuss the variance reduction methods when $F$ has a finite-sum structure.   %Due to that the analysis for the subgradient method is similar to that of the stochastic subgradient method, we just present stochastic subgradient method and the  Nesterov's accelerated method for Lipschitz continuous operator.  
%\vspace*{-0.15in}
\subsection{Using Gradient Descent/ Extragradient Method}
%\vspace*{-0.1in}
The gradient descent method is presented in Algorithm~\ref{alg:GD} and the extragradient method is presented in Algorithm~\ref{alg:EG} and their linear convergence for the gap $\max_{\z\in\Z}F(\bar \z)^{\top}(\bar\z-\z)$ is stated in the following proposition. It is worth mentioning that the last step $\bar\z = \text{Proj}_{\Z}(\w^{(T+1)} - F(\w^{(T+1)})/(\sqrt{2}L))$ in Algorithm~\ref{alg:EG} following Lemma~\ref{lem:prom} is important to prove the linear convergence of the gap $\max_{\z\in\Z}F(\bar \z)^{\top}(\bar\z-\z)$. 

\begin{proposition}\label{prop:gd}
	Suppose that $F_k$ is single-valued and $L_k$-Lipschitz continuous and  is $\mu$-strongly monotone. Algorithm~\ref{alg:GD} applied to SVI$(F_k, \Z)$ with $\eta = \frac{\mu}{2L_k^2}$ guarantees %that for any $\z\in\Z$
	%\small
	\begin{align*}%\label{eqn:ineq2}
	\max_{\z\in\Z}F_k(\z^{(T)})^{\top}(\z^{(T)}-\z)
	\leq   4D^2L_k\beta\exp(-\frac{T-1}{4\beta^2}), \text{ where $\beta= L_k/\mu\geq 1$}.% and $\w_*$ denotes a solution to $\text{SVI}(F_k, \Z)$.
	\end{align*}
	Algorithm~\ref{alg:EG} applied to SVI$(F_k, \Z)$ with $\eta=1/(4L_k)$ guarantees that for any $\z\in\Z$
	\begin{eqnarray*}
		\max_{\z\in\Z}F_k(\bar \z)^{\top}(\bar\z-\z)\leq 4D^2L_k\exp(-T/(8\beta)), \text{	where $\beta = L_k/\mu\geq 1$.}
	\end{eqnarray*}
	
	%and
	%\[
	%$\mu \|\z^{(T)} - \w_*\|^2 \leq  4D^2L_k\beta\exp(-\frac{T-1}{4\beta^2})$
	%\]
	%\normalsize
\end{proposition}
%{\bf Remark:} It is worth mentioning that proving the upper bound of $F(\z^{(T+1)})^{\top}(\z^{(T+1)}-\z)$ or $\E[F(\z^{(\tau)})^{\top}(\z^{(\tau)}-\z)]$ for any $\z\in\Z$  is important for later analysis. 
%\vspace*{-0.1in}
The following two corollaries summarize the convergence results of IPP-GD and IPP-EG, respectively.

\begin{corollary}\label{thm:gd}
	For problem~(\ref{eqn:P}), assume that Assumption~\ref{ass:3} holds and  $F$ is single-valued and $L$-Lipschitz continuous.  Suppose Algorithm~\ref{alg:GD} is used as $\text{ApproxSVI}$ with $\eta_k= \frac{\rho}{2(L+2\rho)^2},  T_k=1+\frac{4(L+2\rho)^2}{\rho^2}\log(\frac{8(L+2\rho)^{2}(k+1)}{\rho^2})$. By choosing  $\gamma=1/(2\rho)$, $\theta_k =(k+1)^\alpha$ with $\alpha\geq 1$, and  $K= \frac{16\rho^2 D^2(\alpha+1)}{\epsilon^2}$ in Algorithm~\ref{alg:meta}, we have
	\begin{align*}%\label{eqn:fc}%\label{eqn:proximal mapping_converge}
	&\E[ \|\z_\tau - \bar\z_\tau \|^2]\leq \gamma^2\epsilon^2, \quad \E[\text{dist}^2(0, \partial (f(\bar\u_\tau, \bar\v_\tau)+ 1_{\Z}(\bar\u_\tau, \bar\v_\tau))]\leq \epsilon^2%\leq \frac{2D^2\sqrt{K}}{(2/3)K^{3/2}} + \frac{4\sum_{k=0}^{K-1}\sqrt{k+1}\varepsilon(\eta_k, T_k)}{(2/3)\rho K^{3/2}}.
	\end{align*}
	where $\bar\z_\tau=(\bar\u_\tau, \bar\v_\tau)^{\top}$ is the solution to $\text{SVI}(F^\gamma_{\z_\tau}, \Z)$. % and
	%	\begin{align}\label{eqn:fc2}%\label{eqn:proximal mapping_converge}
	%	\E[\max_{\z\in\Z}\langle F(\bar \z_\tau), \bar\z_\tau - \bz\rangle]\leq  \epsilon%\leq \frac{2D^2\sqrt{K}}{(2/3)K^{3/2}} + \frac{4\sum_{k=0}^{K-1}\sqrt{k+1}\varepsilon(\eta_k, T_k)}{(2/3)\rho K^{3/2}}.
	%	\end{align}
	The total iteration complexity of $O(\log(\frac{1}{\epsilon})\frac{L^2}{\epsilon^2})$.
\end{corollary}
\begin{corollary}\label{thm:eg}
	Under the same assumption in Corollary~\ref{thm:gd},  Algorithm~\ref{alg:meta} with $\gamma=1/(2\rho)$, $\theta_k =(k+1)^\alpha$ with $\alpha>1$,  $ T_k= \frac{8(L+2\rho)}{\rho}\log(\frac{4(k+1)(L+2\rho)}{\rho})$, and $K=24\rho^2 D^2(\alpha+1)/\epsilon^2$  guarantees
	\begin{align}\label{eqn:fc3}%\label{eqn:proximal mapping_converge}
	\E[ \|\z_\tau - \bar\z_\tau \|^2]\leq \gamma^2\epsilon^2,\quad 	E[\text{dist}^2(0, \partial (f(\bar\u_\tau, \bar\v_\tau)+ 1_{\Z}(\bar\u_\tau, \bar\v_\tau))]\leq \epsilon^2%\leq \frac{2D^2\sqrt{K}}{(2/3)K^{3/2}} + \frac{4\sum_{k=0}^{K-1}\sqrt{k+1}\varepsilon(\eta_k, T_k)}{(2/3)\rho K^{3/2}}.
	\end{align}
	where $\bar\z_\tau$ is the solution to $\text{SVI}(F^\gamma_{\z_\tau}, \Z)$. The total iteration complexity is $O(\log(\frac{1}{\epsilon})\frac{L\rho}{\epsilon^2})$.
\end{corollary}
%\vspace*{-0.1in}
{\bf Remark: }  The improvement of using the extragradient method over the gradient descent method lies at the better dependence on the condition number $L/\rho\geq 1$. In particular, IPP-GD has a complexity of $\widetilde O((L/\rho)^2\rho^2/\epsilon^2)$ and IPP-EG has a complexity of $\widetilde O((L/\rho)\rho^2/\epsilon^2)$.

\begin{figure*}
	\hspace*{-0.12in} \begin{minipage}[t]{0.43\textwidth}
		\begin{algorithm}[H]
			\caption{GD$(F,\Z,\bz^{(0)},\eta,T)$}\label{alg:GD}
			\begin{algorithmic}[1]
				%\STATE \textbf{Input:} single-valued mapping $F$, set $\Z$, $\bz^{(0)}\in\mathcal{Z}$, $T\geq 1$.  
				\FOR{$t=0,..., T-1$}
				\STATE $\z^{(t+1)}= \text{Proj}_{\Z}\left(\z^{(t)} - \eta  F(\z^{(t)})\right)$ 
				%\STATE $\w^{(t+1)}= \text{Proj}_{\Z}\left( \w^{(t)} - \eta  F(\z^{(t)})\right)$ 
				\ENDFOR
				%\STATE Sample $\tau$ uniformly randomly from $\{0,1,\dots,T-1\}$. 
				\STATE \textbf{Output:} $\z^{(T)}$
				%\vspace*{0.06in}
			\end{algorithmic}
		\end{algorithm}
	\end{minipage}
	\hspace*{0.02in} \begin{minipage}[t]{0.6\textwidth}
		\begin{algorithm}[H]
			\caption{%Extragradient method for $\text{SVI}(F, Z)$: 
				EG$(F,\Z,\bw^{(0)},\eta,T)$}\label{alg:EG}
			\begin{algorithmic}[1]
				%\STATE \textbf{Input:} $F$,  $\Z$, $\bw^{(0)}\in\mathcal{Z}$, $\eta = 1/(4L)$ and  $T\geq 1$.  
				\FOR{$t=0,..., T$}
				\STATE $\z^{(t)}= \text{Proj}_{\Z}\left(\w^{(t)} - \eta  F(\w^{(t)})\right)$ 
				\STATE $\w^{(t+1)}= \text{Proj}_{\Z}\left( \w^{(t)} - \eta  F(\z^{(t)})\right)$ 
				\ENDFOR
				\STATE \textbf{Output:} {\small $ \bar\z = \text{Proj}_{\Z}(\w^{(T+1)} -  F(\w^{(T+1)})/(\sqrt{2}L))$}
			\end{algorithmic}
		\end{algorithm}
	\end{minipage}
\end{figure*}
\subsection{Using Variance Reduction methods}
%\vspace*{-0.1in}
% Next, we discuss how to use  variance reduction methods for solving each strongly monotone SVI when the underlying $F(\z)$ is Lipschitz continuous and has a finite-sum form $F(\z) = \frac{1}{n}\sum_{i=1}^n F_i(\z)$. Several studies have considered variance reduction algorithms for solving strongly monotone SVI or the strongly convex and strongly concave min-max problems~\citep{DBLP:conf/nips/PalaniappanB16,hassan18nonconvexmm}. 

%\subsection{Using Variance Reduction Methods}
%\label{sec:svrg}
Next, we discuss how to use  variance reduction methods for solving each strongly monotone SVI when the underlying $F(\z)$ is $\mu$-strongly monotone, has a finite-sum form 
$$
F(\z) = \frac{1}{n}\sum_{i=1}^n F_i(\z),
$$ 
where each $F_i$ is $L$-Lipschitz continuous.
Several studies have considered variance reduction algorithms for solving strongly monotone SVI or the strongly convex and strongly concave min-max problems~\citep{DBLP:conf/nips/PalaniappanB16,hassan18nonconvexmm}. It is notable that in these studies, linear convergence is only proved for the point convergence, i.e., $\E[\|\zh - \w_*\|]$ where $\w_*$ denotes the solution of the strongly monotone SVI. However, by using an additional gradient update as in Lemma~\ref{lem:prom}, we can get the linear convergence for  $\E[\max_{\z\in\Z}F_k(\bar \z)^{\top}(\bar\z-\z)]$.  Combining this result, existing convergence results of variance reduction algorithms for solving each strongly monotone $\text{SVI}(F_k, \Z)$ and the result in Theorem~\ref{thm:meta}, we can derive the complexity for solving the original  $\text{SVI}(F, \Z)$ or the min-max saddle-point problem. According to~\citep{DBLP:conf/nips/PalaniappanB16}, the complexity of SVRG-based algorithm for solving  a strongly convex and strongly concave min-max smooth problem in the sense that $\E[ \|\zh - \w_*\|]\leq \epsilon$ is $O((n + L^2/\mu^2)\log(1/\epsilon))$, which gives a total gradient  complexity of $\widetilde O((n + L^2/\rho^2)/\epsilon^2)$ for finding a solution $\z_\tau$ that satisfies $\E[\|\z_\tau-\bar\z_\tau\|]\leq\epsilon^2$  where $\bar\z_\tau=(\bar\u_\tau, \bar\v_\tau)^{\top}$ is the solution to $\text{SVI}(F^\gamma_{\z_\tau}, \Z)$. Equivalently, for finding a nearly $\epsilon$-stationary solution of the original problem, the total gradient complexity is $\widetilde O((n\rho^2 + L^2)/\epsilon^2)$. 

Now we present the details about the approach. Following the setting in \citep{DBLP:conf/nips/PalaniappanB16}, we first equivalently view the SVI problem as the problem of finding a $\bz$ such that
$$
\mathbf{0}\in F(\z)+\partial\mathbf{1}_{\mathcal{Z}}(\bz),
$$  
where $\partial\mathbf{1}_{\mathcal{Z}}(\bz)=N_{\mathcal{Z}}(\z)$ is the normal cone of $\mathcal{Z}$ at $\bz$.
Note that  $F(\z)+\partial\mathbf{1}_{\mathcal{Z}}(\bz)$ is still a strongly monotone mapping. Furthermore, we reformulate $F(\z)+\partial\mathbf{1}_{\mathcal{Z}}(\bz)$  as 
$$
F(\z)+\partial\mathbf{1}_{\mathcal{Z}}(\bz)=A(\z)+B(\bz)=A(\z)+\sum_{i=1}^n B_i(\z),
$$
where $A(\bz):=\mu(\bz-\bw)+\partial\mathbf{1}_{\mathcal{Z}}(\bz)$ is $\mu$-strongly monotone, $B(\bz):=F(\bz)-\mu(\bz-\bw)$ is monotone and $(L+\mu)$-Lipschitz-continuous,
and $B_i(\bz):=\frac{1}{n}F_i(\bz)-\frac{\mu}{n}(\bz-\bw)$ is $(L+\mu)/n$-Lipschitz-continuous for $i=1,\dots,n$. Based on this structure, we present the SVRG algorithm for SVI by \cite{DBLP:conf/nips/PalaniappanB16} in Algorithm~\ref{alg:SVRG}. Note that Algorithm~\ref{alg:SVRG} utilizes the resolvent operator of the mapping $A$ defined above, which is the mapping $(I+\eta A)^{-1}(\bz)=\argmin_{\bz'\in\mathcal{Z}}\frac{1}{2\eta}\|\bz'-\bz\|^2+\frac{\mu}{2}\|\bz'-\bw\|^2$.

%When $f(\bx,\by)=\frac{1}{n}\sum_{i=1}^nf_i(\bx,\by)$, the mapping $F_{\bw}^\gamma(\bz)$ defined in \eqref{eq:ppm} satisfies the assumption above with $A(\bz)=\left(\frac{1}{\gamma}-\rho\right)(\bz-\bw)+\partial\mathbf{1}_{\mathcal{Z}}(\bz)$, $B_i(\bz)= \frac{1}{n}(\Delta_x f(\x, \y), -\Delta_yf(\x, \y))+\frac{\rho}{n}(\bz-\bw)$, $\mu=\frac{1}{\gamma}-\rho$, and $L=L+\rho$. 
\begin{algorithm}[tb]
	\caption{SVRG Method for $\text{SVI}(F, Z)$: SVRG$(F,\Z,\bw^{(0)}, T)$}\label{alg:SVRG}
	\begin{algorithmic}[1]
		\STATE \textbf{Input:} $\mu-$Strongly Monotone and $L$-Lipschitz continuous mapping $F=\frac{1}{n}\sum_{i=1}^n F_i(\z)$, set $\Z$, $\bw^{(0)}\in\mathcal{Z}$,  and an integer $T\geq 1$.  
		\STATE Initialize $\eta = \mu/(4(L+\mu)^2)$, $S= \left\lceil4\log(4)(L+\mu)^2/\mu^2\right\rceil$,
		$A(\bz)=\mu(\bz-\bw)+\partial\mathbf{1}_{\mathcal{Z}}(\bz)$, $B(\bz)=F(\bz)-\mu(\bz-\bw)$, and $B_i(\bz)=\frac{1}{n}F_i(\bz)-\frac{\mu}{n}(\bz-\bw)$ for $i=1,\dots,n$.
		\FOR{$t=0,..., T-1$}
		\STATE $\bar B^{(t)}=B(\bw^{(t)})$ and $\bz^{(0)}=\bw^{(t)}$
		\FOR{$s=0,..., S-1$}
		\STATE Sample $i$ uniformly randomly from $\{1,2,\dots,n\}$.
		\STATE $B^{(s)}=\bar B^{(t)}-n B_i(\bw^{(t)})+n B_i(\bz^{(s)})$
		\STATE $\z^{(s+1)}= (I+\eta A)^{-1}(\bz^{(s)}-\eta B^{(s)})$ 
		\ENDFOR
		\STATE $\bw^{(t+1)}=\z^{(S)}$
		\ENDFOR
		\STATE \textbf{Output:} $ \bar\z= \text{Proj}_{\Z}\left(\w^{(T)} - F(\w^{(T)})/(\sqrt{2}L)\right)$
	\end{algorithmic}
\end{algorithm}

\begin{proposition}\label{prop:svrg}When $F(\bz)$ is single-valued, $\mu$-strongly monotone, and of the form $F=\frac{1}{n}\sum_{i=1}^n F_i(\z)$ with each $F_i(\z)$  $L$-Lipschitz continuous,  Algorithm~\ref{alg:SVRG} guarantees that for any $\z\in\Z$
	\begin{eqnarray}\label{eqn:ineq2svrg}
	\mathbb{E}[\max_{\z\in\Z}F(\bar\z)^{\top}(\bar\z - \z)] \leq D^2L(2+ \sqrt{2})(\sqrt{3}/2)^{T}.
	\end{eqnarray}
	In addition, we have
	\[
	\mu \|\bar\z - \w_*\|^2 \leq D^2L(2+ \sqrt{2})(\sqrt{3}/2)^{T}
	\]
	where $\w_*$ denotes a solution to $\text{SVI}(F, \Z)$.
	%, and $M = \max_{\z, \w\in\Z}F(\w)^{\top}(\z - \w) + \frac{\mu}{2}\|\z -\w\|^2$.
\end{proposition}
\begin{proof}
	According to Section D.1 in \citep{DBLP:conf/nips/PalaniappanB16}, with the given $S$ and $\eta$,  Algorithm~\ref{alg:SVRG} guarantees that 
	$$
	\mathbb{E}\|\bw^{(T)}-\w_*\|^2\leq\left(\frac{3}{4}\right)^{T}\|\bw^{(0)}-\w_*\|^2,
	$$
	where $\w_*$ denotes a solution to $\text{SVI}(F, \Z)$ (so that $\mathbf{0}\in F(\w_*)+\partial\mathbf{1}_{\mathcal{Z}}(\w_*)$).
	Following Lemma~\ref{lem:prom}, we have
	\begin{align*}
	\mathbb{E}[\max_{\z\in\Z}F(\bar\z)^{\top}(\bar\z - \z)] \leq DL(2+ \sqrt{2})\mathbb{E}\|\bw^{(T)} - \w_*\|\leq D^2L(2+ \sqrt{2})(\sqrt{3}/2)^{T}
	\end{align*}
\end{proof}
Note that the complexity of Algorithm~\ref{alg:SVRG} for $T$ iterations is $O(T(n+L^2/\mu^2))$. 
\begin{corollary}\label{thm:svrg}
	Suppose Assumption~\ref{ass:3}  holds and $F(\bz)$ is single-valued and of the form $F=\frac{1}{n}\sum_{i=1}^n F_i(\z)$ with each $F_i(\z)$  $L$-Lipschitz continuous,  and Algorithm~\ref{alg:SVRG} is used to implement $\text{ApproxSVI}$. Algorithm~\ref{alg:meta} with $\gamma=1/(2\rho)$, $\theta_k =(k+1)^\alpha$ with $\alpha>1$,  $ T_k= (1-\sqrt{3}/2)^{-1}\log((2+\sqrt{2})L(k+1)/\rho)$, and a total of stages $K=6D^2(\alpha+1)/\epsilon^2$  guarantees
	\begin{align}\label{eqn:fc3svrg}%\label{eqn:proximal mapping_converge}
	\E[ \|\z_\tau - \bar\z_\tau \|^2]\leq \epsilon^2,%\leq \frac{2D^2\sqrt{K}}{(2/3)K^{3/2}} + \frac{4\sum_{k=0}^{K-1}\sqrt{k+1}\varepsilon(\eta_k, T_k)}{(2/3)\rho K^{3/2}}.
	\end{align}
	and
	\begin{align}\label{eqn:fc4svrg}%\label{eqn:proximal mapping_converge}
	\E[\text{dist}^2(0, \partial (f(\bar\u_\tau, \bar\v_\tau)+ 1_{\Z}(\bar\u_\tau, \bar\v_\tau))]\leq \epsilon^2/\gamma^2,
	%\E[\max_{\z\in\Z}\langle F(\bar \z_\tau), \bar\z_\tau - \bz\rangle]\leq  2D\rho\epsilon%\leq \frac{2D^2\sqrt{K}}{(2/3)K^{3/2}} + \frac{4\sum_{k=0}^{K-1}\sqrt{k+1}\varepsilon(\eta_k, T_k)}{(2/3)\rho K^{3/2}}.
	\end{align}
	where $\bar\z_\tau$ is the solution to $\text{SVI}(F^\gamma_{\z_\tau}, \Z)$. The total iteration complexity is $\widetilde O(D^2(n+L^2/\rho^2)/(\epsilon^2))$.
\end{corollary}

%\vspace*{-0.1in}

\section{Experimental Results}
%\vspace*{-0.1in}
\subsection{Synthetic Experiment}
We conduct two synthetic experiments to validate the presented theoretical results. 

The first experiment focuses on the comparison between IPP-GD and IPP-EG. To this end, we consider
$f(\x,\y) = \frac{1}{2}\x^\top A\x+\x^\top \y-\frac{1}{2}\y^\top A\y$,  $\X=\{(x_1,x_2)\in\R^2:0\leq x_2\leq x_1\leq 1\}$, $\Y=\{(y_1,y_2)\in\R^2:0\leq y_2\leq y_1\leq 1\}$,  where  $A=\text{diag}(1,-\rho)$, $0<\rho<1$. It is not difficult to show that Assumption \ref{ass:3} holds, i.e., $D=\sqrt{2}$, the mapping $F(\x,\y)=(\partial_{\x}f(\x,y), -\partial_{\y}f(\x,\y))^\top $ is $\rho$-weakly monotone and the corresponding MVI problem has a solution $(0,0,0,0)$. In addition, $F(\x,\y)$ is Lipschitz continuous with modulus $L=1$.  For both algorithms, we start from $(0.2,0.1,0.2,0.1)$, set the hyper-parameters of two algorithms according to Corollary~\ref{thm:gd} and~\ref{thm:eg} respectively, run both algorithms until reaching a point whose  gradient's magnitude  is less than $10^{-2}$, and report the number of gradient evaluations in Table~\ref{ex:table2}. % within distance $10^{-2}$ to the optimal solution $(0,0,0,0)$ (with result reported in Table~\ref{ex:table2}). %The metric we use for comparison the number of iterations is the Euclidean distance to the optimal solution $(0,0,0,0)$. 
To illustrate the effect of condition number on the convergence, we set different values of $\rho$ in a range $\{10^{-7:1:-1}\}$. %We also use the norm of gradient as an alternative criterion, i.e., we terminate the algorithm when the norm of gradient is less than $10^{-2}$ (with result reported in Table~\ref{ex:table2}).  
From Table~\ref{ex:table2}, we can see that both IPP-GD and IPP-EG converge, and IPP-EG converges much faster. It is worth mentioning that when $\rho$ is sufficiently small (i.e. the condition number $L/\rho$ is sufficiently large), IPP-EG is $L/\rho$ times faster than IPP-GD, which is consistent with our theory. 

In the second experiment, We aim to verify the theory of IPP-SGD and IPP-SVRG. We consider the following finite sum min-max optimization problem.
$$\min_{\x\in\X}\max_{\y\in\Y}f(\x,\y) = \frac{1}{n}\sum_{i=1}^{n}\xi_i\left[\frac{1}{2}\x^\top A\x+\x^\top \y-\frac{1}{2}\y^\top A\y\right],$$  
where $\xi_i\sim N(1, 5)$, $i=1,\ldots, n-1$, $n=1000$, and $\xi_n=n-\sum_{i=1}^{n-1}\xi_i$, $\X=\{(x_1,x_2)\in\R^2:0\leq x_2\leq x_1\leq 1\}$, $\Y=\{(y_1,y_2)\in\R^2:0\leq y_2\leq y_1\leq 1\}$,  $A=\text{diag}(1,-\rho)$, $0<\rho<1$. We can also show that Assumption \ref{ass:3} holds, i.e., $D=\sqrt{2}$, the mapping $F(\x,\y)=(\partial_{\x}f(\x,y), -\partial_{\y}f(\x,\y))^\top $ is $\rho$-weakly monotone and the corresponding MVI problem has a solution $(0,0,0,0)$. In addition, $f(\x,\y)$ is $L$-Lipschitz continuous with modulus $1$. For IPP-SGD and IPP-SVRG, we start from $(0.2,0.1,0.2,0.1)$, run both algorithms until reaching a point whose gradient's magnitude is less than a certain threshold $\epsilon$. We consider two different values of $\epsilon$ ($10^{-10}$ and $10^{-1}$) and report the corresponding number of stochastic gradient evaluations in Table~\ref{ex:table3}. We can see that both algorithms converge. IPP-SVRG converges faster than IPP-SGD when the target accuracy is sufficiently small (e.g., $10^{-10}$), and IPP-SGD converges faster than IPP-SGD when the target accuracy is not small enough (e.g. $10^{-1}$). This observation is consistent with our theory.
\begin{table}[t]
	%\vspace*{0.05in}
	%\vspace*{-0.05in}
	\label{ex:table2}
	\centering
	\scalebox{1}{
		\begin{tabular}{|c|c|c|c|c|c|c|c|}
			\hline
			$\rho$& $10^{-1}$       & $10^{-2}$       & $10^{-3}$       & $10^{-4}$      & $10^{-5}$      & $10^{-6}$      & $10^{-7}$            \\\hline
			IPP-GD & $9239$ & $500334$  & $14857$  & $144579$ & $1441992$ & $14416139$              & $>10^8$                          \\\hline
			IPP-EG & $1854$                  & $10942$ & $62$ & $60$ & $60$ & $60$ & $60$ \\\hline
	\end{tabular}}
	\caption{IPP-EG vs IPP-GD for the synthetic experiment. The numbers indicate the number of gradient evaluations to achieve a point where the gradient's norm is less than $10^{-2}$.
	}
\end{table}
\begin{table}
	\centering
	%\scalebox{1}{
	\begin{tabular}{|c|c|c|c|}
		\hline
		$\rho$ ($\epsilon=10^{-10}$)& $10^{-1}$       & $10^{-2}$       & $10^{-3}$    \\\hline  % & $10^{-4}$       \\\hline
		%	IPP-SGD & $3339092$ & $231784$  & $86243$  \\\hline
		%& $62007$      
		%\\\hline
		IPP-SGD & $12710$ & $3051$  & $2515$  \\\hline
		IPP-SVRG & $1002$                  & $1006$ & $1503$ \\\hline
		%& $12486319$  \\\hline
	\end{tabular}
	%}
	\begin{tabular}{|c|c|c|c|}
		\hline
		$\rho$ ($\epsilon=10^{-1}$)& $10^{-1}$       & $10^{-2}$       & $10^{-3}$    \\\hline  % & $10^{-4}$       \\\hline
		%	IPP-SGD & $3339092$ & $231784$  & $86243$  \\\hline
		%& $62007$      
		%\\\hline
		IPP-SGD & $172$ & $70$  & $67$  \\\hline
		IPP-SVRG & $1002$                  & $1002$ & $1025$ \\\hline
		%& $12486319$  \\\hline
	\end{tabular}
	\caption{IPP-SGD vs IPP-SVRG for the synthetic  experiment. The left (right) table is the number of stochastic gradient evaluations to achieve a point where the gradient's norm is less than $10^{-10}$ ($10^{-1}$).%The numbers indicate the number of gradient evaluations to achieve a point where the gradient's norm is less than $10^{-1000}$.
	}
%	\vspace*{-0.05in}
	\label{ex:table3}
\end{table}
%\begin{figure*}[t]
%	\hspace*{-0.1in}\includegraphics[scale=0.18]{./ICML_GAN_result/cifar10/ELU/wgan_sgd_cifar10_elu.eps}
%	\hspace*{-0.1in}\includegraphics[scale=0.18]{./ICML_GAN_result/cifar10/ELU/wgan_adam_cifar10_elu.eps}
%	\hspace*{-0.08in}\includegraphics[scale=0.18]{./ICML_GAN_result/cifar10/ReLU/wgan_sgd_cifar10_relu.eps}
%	\hspace*{-0.05in}\includegraphics[scale=0.18]{./ICML_GAN_result/cifar10/ReLU/wgan_adam_cifar10_relu.eps}
%	%\includegraphics[scale=0.2]{./ICML_GAN_result/cifar10/ELU/wgangp_sgd_cifar10_elu.eps}
%	
%	\hspace*{-0.1in}\includegraphics[scale=0.18]{./ICML_GAN_result/cifar10/ReLU/wgangp_sgd_cifar10_relu.eps}
%	\hspace*{-0.1in}\includegraphics[scale=0.18]{./ICML_GAN_result/cifar10/ReLU/wgangp_adam_cifar10_relu.eps}
%	\hspace*{-0.1in}	\includegraphics[scale=0.18]{./ICML_GAN_result/bedroom/ReLU/wgan_sgd_lsun_relu.eps}
%	%\hspace*{-0.1in}	\includegraphics[scale=0.18]{./ICML_GAN_result/bedroom/ELU/wgan_sgd_lsun_elu.eps}
%	\hspace*{-0.1in}	\includegraphics[scale=0.18]{./ICML_GAN_result/bedroom/ReLU/wgan_adam_lsun_relu.eps}
%	%\hspace*{-0.1in}	\includegraphics[scale=0.18]{./ICML_GAN_result/bedroom/ELU/wgan_adam_lsun_elu.eps}
%	%\includegraphics[scale=0.2]{./ICML_GAN_result/cifar10/ELU/wgangp_adam_cifar10_elu.eps}
%%	\vspace*{-0.1in}
%	\caption{Comparison of different methods for solving WGAN and WGAN-GP.}
%	\label{exp:cifar10}
%	%\vspace*{-0.15in}
%\end{figure*}
\begin{figure*}[t]
	\hspace*{-0.1in}\includegraphics[scale=0.18]{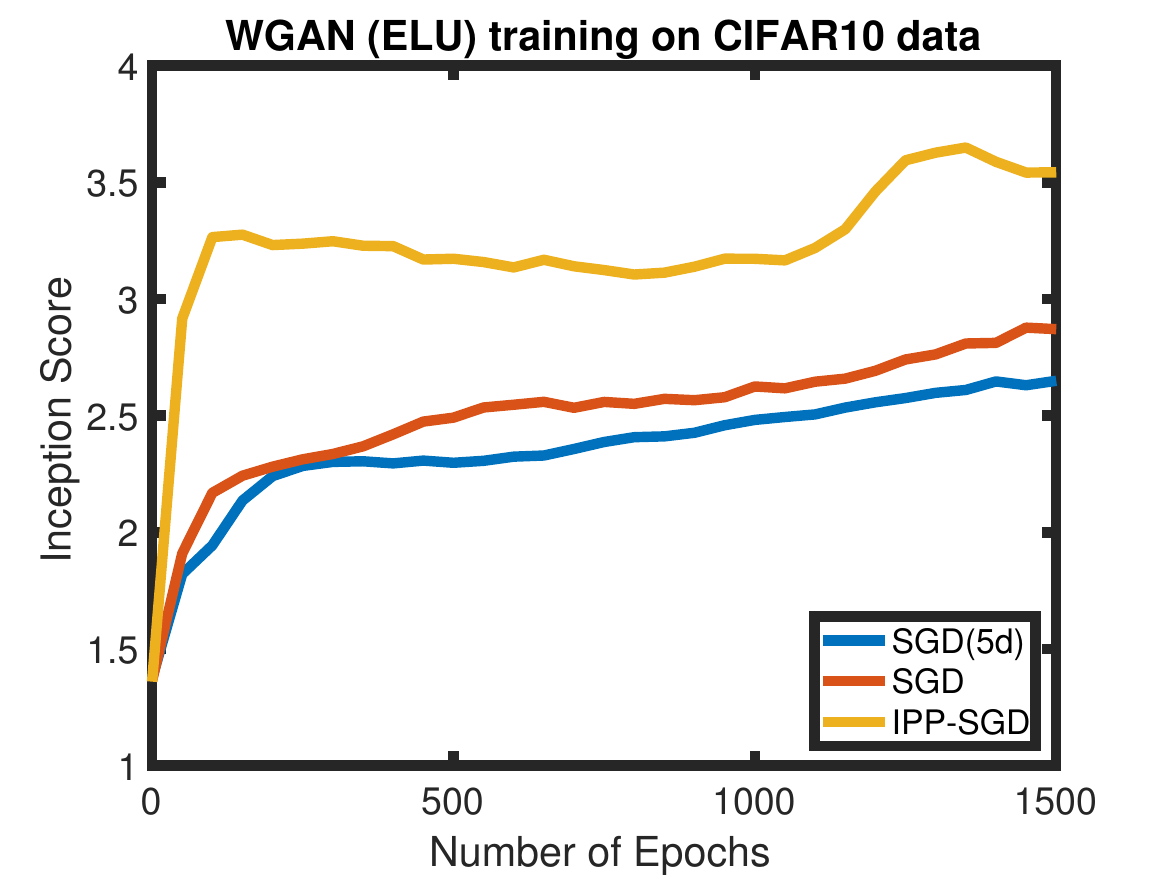}
	\hspace*{-0.1in}\includegraphics[scale=0.18]{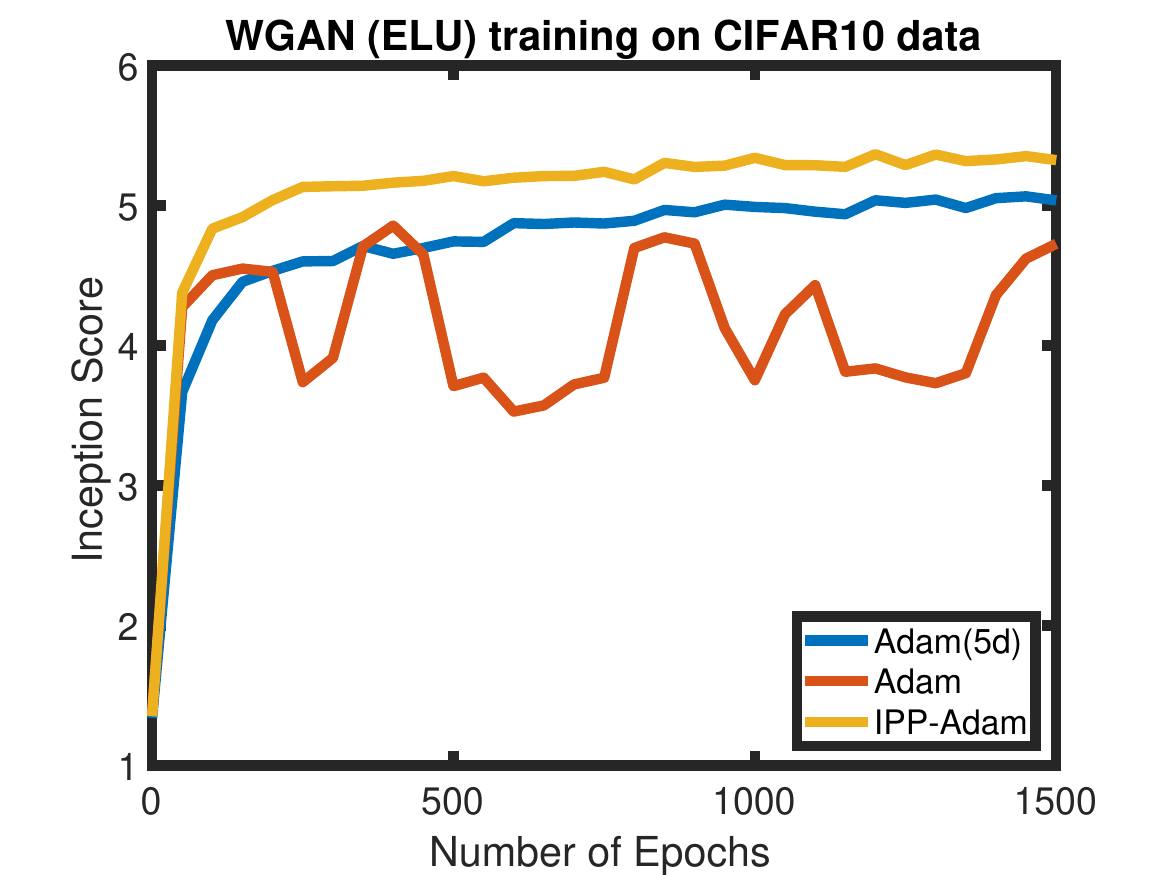}
	\hspace*{-0.08in}\includegraphics[scale=0.18]{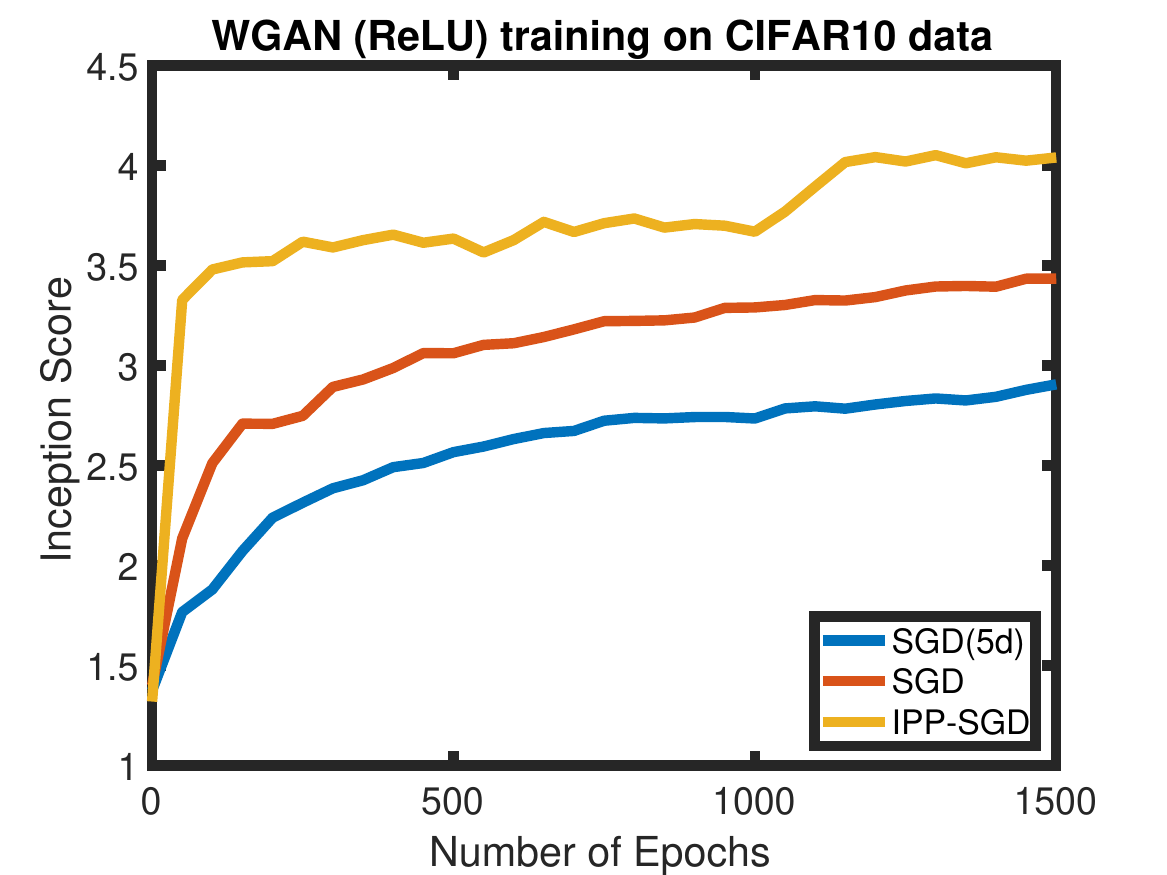}
	\hspace*{-0.05in}\includegraphics[scale=0.18]{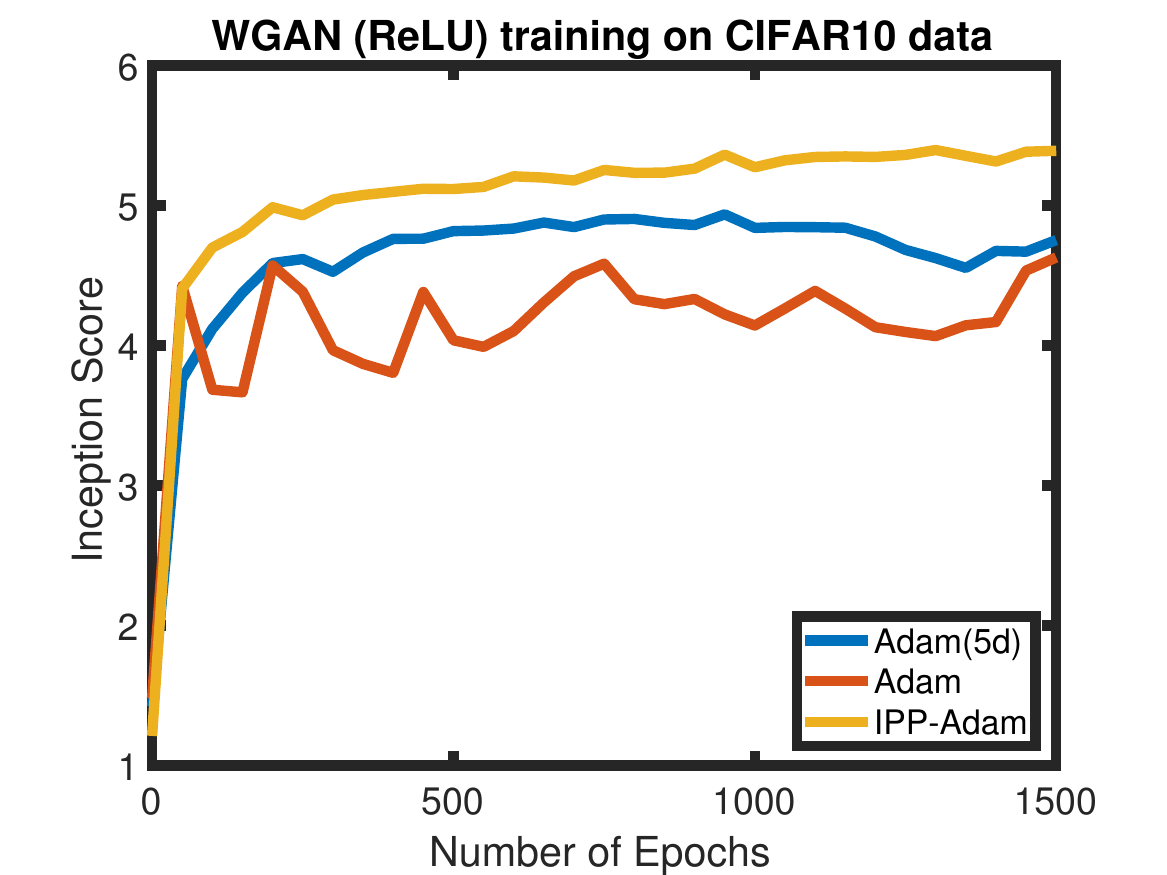}

	\hspace*{-0.1in}\includegraphics[scale=0.18]{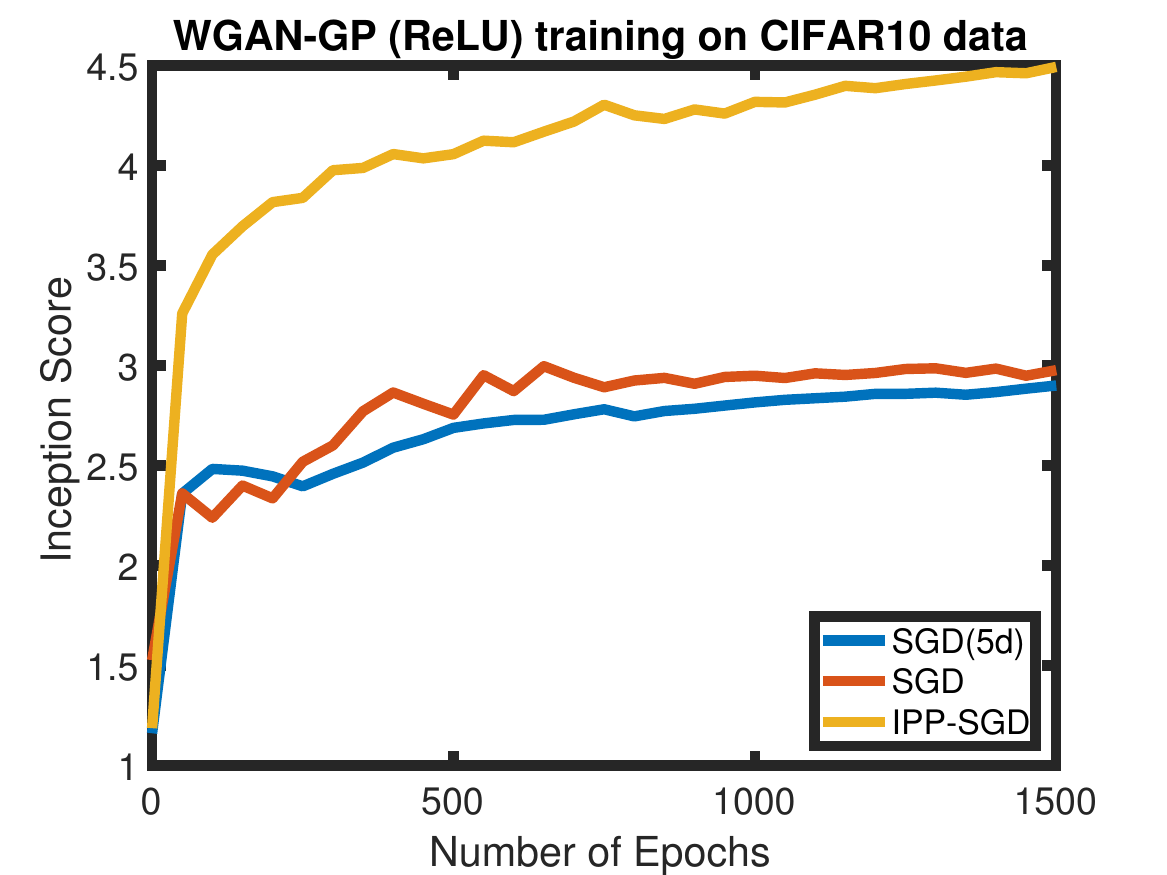}
	\hspace*{-0.1in}\includegraphics[scale=0.18]{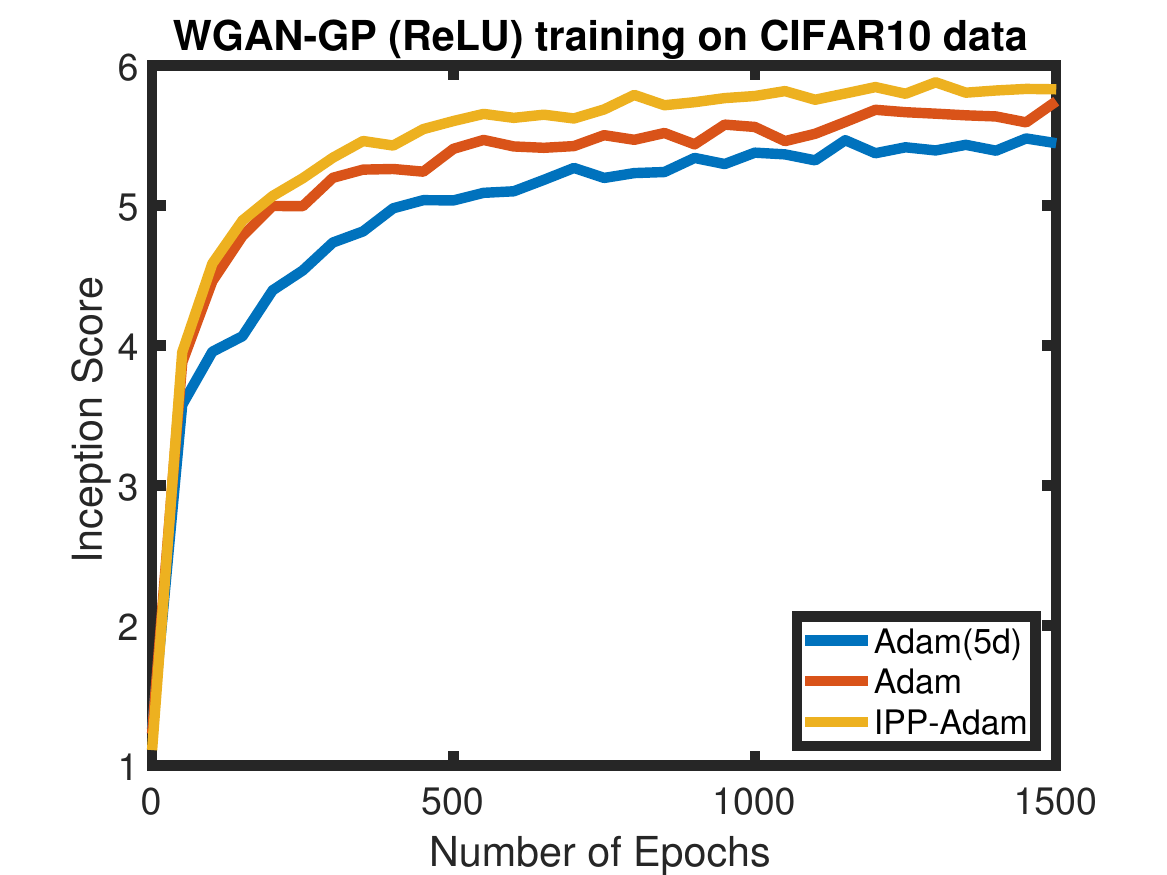}
	\hspace*{-0.1in}	\includegraphics[scale=0.18]{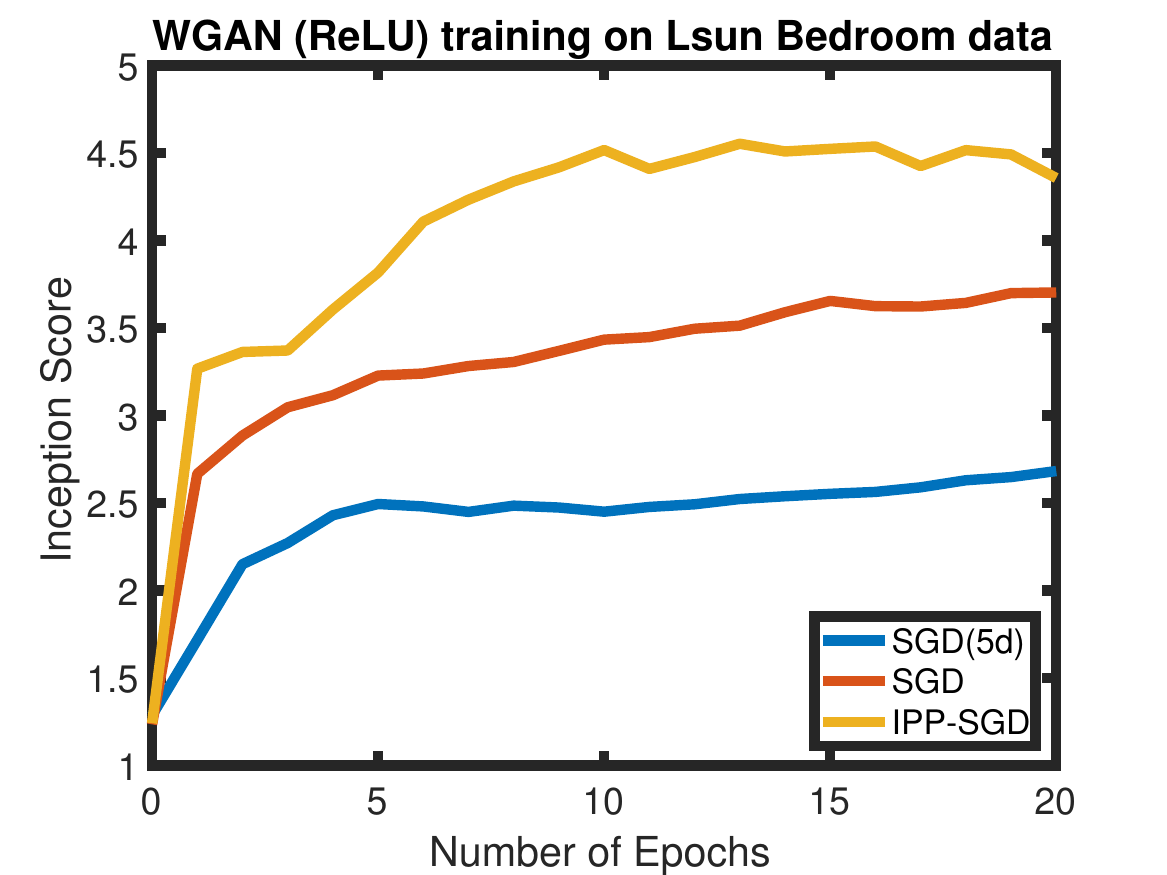}
	%\hspace*{-0.1in}	\includegraphics[scale=0.18]{./ICML_GAN_result/bedroom/ELU/wgan_sgd_lsun_elu.eps}
	\hspace*{-0.1in}	\includegraphics[scale=0.18]{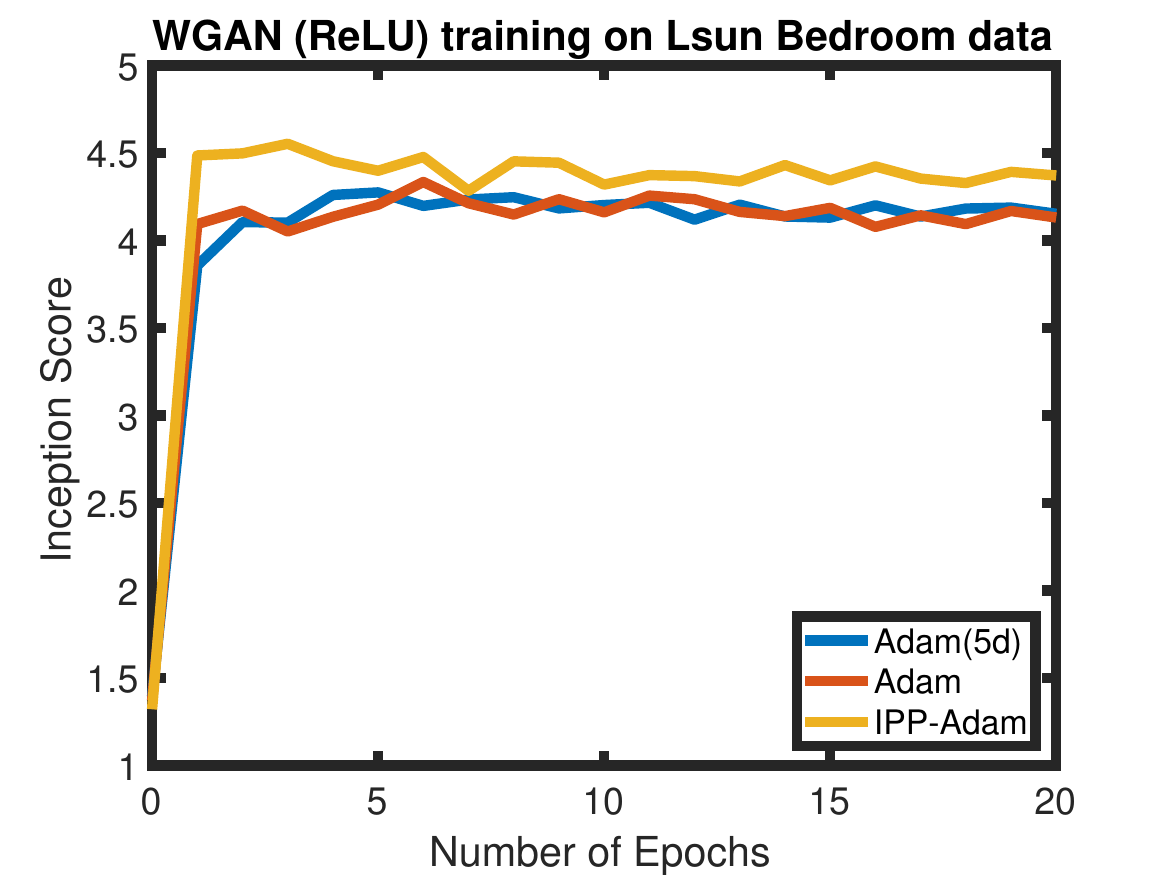}
	%\hspace*{-0.1in}	\includegraphics[scale=0.18]{./ICML_GAN_result/bedroom/ELU/wgan_adam_lsun_elu.eps}
	%\includegraphics[scale=0.2]{./ICML_GAN_result/cifar10/ELU/wgangp_adam_cifar10_elu.eps}
	%	\vspace*{-0.1in}
	\caption{Comparison of different methods for solving WGAN and WGAN-GP.}
	\label{exp:cifar10}
	%\vspace*{-0.15in}
\end{figure*}
% \begin{figure}[t]
% 	\centering
% 	%\includegraphics[scale=0.2]{./ICML_GAN_result/cifar10/ELU/wgangp_sgd_cifar10_elu.eps}
% 	\hspace*{-0.1in}	\includegraphics[scale=0.2]{./ICML_GAN_result/bedroom/ELU/wgan_sgd_lsun_elu.eps}
% 	\hspace*{-0.1in}	\includegraphics[scale=0.2]{./ICML_GAN_result/bedroom/ELU/wgan_adam_lsun_elu.eps}
% 	%\includegraphics[scale=0.2]{./ICML_GAN_result/cifar10/ELU/wgangp_adam_cifar10_elu.eps}
% %	\vspace*{-0.1in}
% 	\caption{Comparison of different methods for solving WGAN with ELU activation on LSUN Bedroom data.}
% 	\label{exp:cifar10_1}
% \end{figure}

\subsection{Training GANs}
Next, we report some experimental results to justify the effectiveness of the proposed algorithm for training GANs. Although there are hundreds of variants of GAN formulations, we focus on two popular variants, namely WGAN~\citep{pmlr-v70-arjovsky17a}  and WGAN with gradient penalty (WGAN-GP)~\citep{DBLP:conf/nips/GulrajaniAADC17}. We conduct the experiments using the same neural network structures for the discriminator and generator as in the original implementations on two datasets (CIFAR10 and LSUN Bedroom). 
We use two activation functions in the networks, namely ReLU and ELU~\citep{clevert2015fast}, with the latter being a smooth function which is consistent with our assumption. 
%Furthermore, we also conduct experiment with another activation function called Exponential Linear Unit (ELU)~\citep{clevert2015fast}. Built upon the original network structures, we replace all ReLUs and LeakyReLUs with ELUs. %Especially for WGAN-GP, we change one term of the gradient penalty from the Euclidean norm of the gradient to the square of it and don't change other hyperparameters. The reason why we consider ELU and the change of the gradient penalty term for WGAN-GP is that it makes the objective smooth in terms of both primal and dual variables, which is consistent with our theory.

We consider two stochastic methods for implementing the subroutine ApproxSVI, namely primal-dual SGD and primal-dual Adam, and refer to our algorithms as IPP-SGD and IPP-Adam. For both of them, the update for the discriminator parameters  and for the generator parameters is conducted simultaneously. We compare IPP-SGD with two baselines, namely SGD (5d) and SGD, and compare IPP-Adam with two baselines, namely,  Adam (5d) and Adam.  5d means that 5 steps of updates for the  discriminator parameters is conducted before 1 step update of the generator parameters, which is considered in their original papers~\citep{pmlr-v70-arjovsky17a,DBLP:conf/nips/GulrajaniAADC17}. In SGD and Adam only 1 step of update is performed for the discriminator and the generator simultaneously.
%before 1 step of update for the generator. %IPP-SGD and IPP-Adam are the proposed algorithms, in which inner loop calls SGD and Adam respectively. 

% We consider two types of activation functions, ELU and ReLU. The reason why we consider ELU is that using ELU makes the objective smooth in terms of both primal and dual variables, which is consistent with our theory. For ReLU case, we follow the neural network structure in the original implementations. For ELU case, we replace all ReLU in the original implementations with ELU.

The performance of algorithms is evaluated by using the inception score~\citep{salimans2016improved} of the generated images. The step size of SGD and SGD (5d) is set as $\eta_0/\sqrt{t}$ with $t$ being the number of performed updates as in standard theory for convex-concave problems.  For IPP-SGD, the step size and the number of inner iterations at $k$-th stage are set to be $\eta_0/(k+1)$ and $T_0(k+1)^2$ according to our theory, respectively. For IPP-Adam, the step size and the number of iterations at $k$-th stage are set to be $\eta_0/(k+1)$ and a large value $T$ for simplicity, respectively. All parameters including $\eta_0, T_0, T, \gamma$ are tuned for each algorithm separately to obtain the best performance. 
%The details for tuning these parameters are reported in the supplement. 
\paragraph{CIFAR10}
For CIFAR10 data,  we tune the initial stepsize of SGD(5d) from $\{10^{-1},5\times 10^{-2}, 10^{-2}, 5\times 10^{-3}, 10^{-3}, 5\times 10^{-4}, 10^{-4}\}$, choose the one with best performance, and then use the same initial stepsize for SGD and IPP-SGD. For SGD(5d) and SGD, the stepsize at iteration $t$ is set to be $\eta_0/\sqrt{t}$, where $\eta_0$ is initial stepsize. For IPP-SGD, the stepsize and the number of inner iterations at $k$-th stage are set to be $\eta_0/(k+1)$ and $T_0(k+1)^2$ respectively, where $T_0$ is tuned from $\{5000:5000:60000\}$. In our experiment on CIFAR10 data, the tuned initial stepsize is $0.1$ for WGAN (ReLU and ELU) and $0.01$ for WGAN-GP (ReLU). For Adam(5d), Adam and IPP-Adam, we choose the initial stepsize according to the original implementations for WGAN and WGAN-GP. For IPP-Adam, the stepsize and the number of passes of CIFAR10 data (with batchsize 64) at $k$-th stage are set to be $\eta'_0/(k+1)$ and $c'$ respectively, where $\eta'_0$ is the initial stepsize for Adam (5d) and $c'$ is tuned from $\{50:50:400\}$. Note that the number of iterations for each stage is $T=50000c'/64$. For our proposed methods (IPP-Adam and IPP-SGD), we tune the weakly monotone parameter from $\{10^{-3}, 10^{-4}\}$. 
%We run the experiments for 1500 passes of CIFAR10 data, and plot the inception score versus the number of epochs (data passes) in Figure \ref{exp:cifar10}. The experimental results show that our proposed algorithms (IPP-Adam and IPP-SGD) have better empirical performance.
%
%
\paragraph{LSUN Bedroom} For LSUN Bedroom dataset, we use the same strategy as we used in CIFAR10 data for SGD(5d), SGD and IPP-SGD, Adam(5d) and Adam. The tuned initial step size for SGD(5d) is $0.01$ for WGAN (ReLU and ELU). For IPP-Adam, we use the same initial stepsize as Adam(5d) and tune the number of inner iterations in a different range. Specifically, it is tuned from $\{5000:5000:50000\}$.

%\begin{figure}[t]
%\includegraphics[scale=0.2]{./ICML_GAN_result/cifar10/ReLU/wgangp_sgd_cifar10_relu.eps}
%\includegraphics[scale=0.2]{./ICML_GAN_result/cifar10/ReLU/wgangp_adam_cifar10_relu.eps}
%\vspace*{-0.1in}
%	\caption{Results for solving WGAN-GP on CIFAR-10 data.}
%	\label{exp:bedroom}
%\end{figure}

The inception score vs the number of epochs passing the data for solving WGAN and WGAN-GP on the two datasets are plotted in Figure~\ref{exp:cifar10}. %and Figure \ref{exp:cifar10_1}. %, and for solving WGAN-GP on CIFAR10 data is plotted in Figure~\ref{exp:bedroom}. 
We can see that in all cases, the proposed IPP methods outperform their plain versions used in previous studies. We also notice that using 5 steps of updates for the discriminator before 1 step of update for the generator is observed to be only effective for training WGAN by Adam.  This strategy is not that effective for training WGAN-GP or SGD. We also notice that using ADAM in IPP converges faster than using SGD in IPP. 

% \paragraph{More Experimental Results} We report more experimental results for WGAN training with ELU activation on LSUN bedroom dataset as illustrated in Figure~\ref{exp:cifar10_1}. We can see that IPP-SGD and IPP-Adam are better than other baselines.

\section{Conclusion}
In this paper, we have presented the first non-asymptotic convergence result for finding a nearly stationary point of non-convex non-concave  saddle-point problems. Our analysis is built on the tool of variational inequalities. An inexact proximal point method is presented with different variations that employ different algorithms for solving the constructed strongly monotone variational inequalities. Synthetic experiments verify our theory and experiments on training two variants of generative adversarial networks demonstrate the effectiveness of the proposed methods. 

\section{Acknowledgments}
We would like to thank anonymous reviewers for their helpful comments. This work was partially supported by NSF  $\#1933212$, NSF CAREER Award $\#1844403$. 
\appendix
\newpage
\section{Proofs of Lemmas}
\label{sec:proofslemmas}
\begin{proof}[Proof of Lemma~\ref{lem:weaklymonotone}]
	Suppose  $F(\bz)=(\partial_x f(\bx,\by),\partial_y[-f(\bx,\by)])$ is $\rho$-weakly monotone. Let $\bz=(\bx,\by)$ and $\bz'=(\bx',\by)$.
	By definition of weak monotonicity, we have for any $\by\in\Y$,  any $\bx,\bx'\in\X$, any $\bxi_x\in\partial_x f(\bx,\by)$, any $\bxi_x'\in\partial_x f(\bx',\by)$, any $-\bxi_y\in\partial_y[- f(\bx,\by)]$, any $-\bxi_y'\in\partial_y[- f(\bx',\by)]$
	$$
	\left\langle\bxi_x-\bxi_x',\bx-\bx'\right\rangle=\left\langle\bxi_x-\bxi_x',\bx-\bx'\right\rangle-\left\langle\bxi_y-\bxi_y',\by-\by\right\rangle=\left\langle\bxi-\bxi',\bz-\bz'\right\rangle\geq -\rho\|\bx-\bx'\|_2
	$$
	where $\bxi=(\bxi_x,-\bxi_y)$ and  $\bxi'=(\bxi_x',-\bxi_y')$.
	This implies that $f(\bx,\by)$ is $\rho$-weakly convex in $\bx$ for any $\by\in\Y$. Similarly, we can show that $f(\bx,\by)$ is $\rho$-weakly concave in $\by$ for any $\bx\in\X$.
	
	Suppose $f(\bx,\by)$ is $\rho$-weakly-convex-weakly-concave. Now let $\bxi_x\in\partial_x f(\bx,\by)$, any $\bxi_x'\in\partial_x f(\bx',\by')$, any $-\bxi_y\in\partial_y[- f(\bx,\by)]$, any $-\bxi_y'\in\partial_y[- f(\bx',\by')]$.  We have 
	\begin{eqnarray*}
		\left\langle\bxi_x-\bxi_x',\bx-\bx'\right\rangle &\geq& f(\bx,\by)-f(\bx',\by)+f(\bx',\by')-f(\bx,\by')-\rho\|\bx-\bx'\|^2\\
		\left\langle\bxi_y'-\bxi_y,\by-\by'\right\rangle &\geq& -f(\bx,\by)+f(\bx,\by')-f(\bx',\by')+f(\bx',\by)-\rho\|\by-\by'\|^2
	\end{eqnarray*}
	for any $\bx,\bx'\in\X$, any $\by,\by\in\Y$,  any $\bxi_x\in\partial_x f(\bx,\by)$, any $\bxi_x'\in\partial_x f(\bx',\by')$, any $-\bxi_y\in\partial_y[- f(\bx,\by)]$, any $-\bxi_y'\in\partial_y[- f(\bx',\by')]$.
	Adding these two inequalities together, we have
	\begin{eqnarray*}
		\left\langle\bxi-\bxi',\bz-\bz'\right\rangle\geq -\rho\|\bx-\bx'\|_2
	\end{eqnarray*}
	where $\bxi=(\bxi_x,-\bxi_y)$ and  $\bxi'=(\bxi_x',-\bxi_y')$. This  means $F(\bz)$ is $\rho$-weakly monotone.
\end{proof}

\begin{proof}[Proof of Lemma~\ref{lem:PPM}]
	Given any $\bz$ and $\bz'$ in $\mathcal{Z}$, any $\bxi\in F(\bz)$ and any $\bxi'\in F(\bz')$, we have 
	\small
	\begin{eqnarray*}
		&&\left\langle \bxi+\frac{1}{\gamma}(\bz-\bw)-\bxi'-\frac{1}{\gamma}(\bz'-\bw), \bz-\bz'\right\rangle\\
		&\geq& \left\langle \bxi-\bxi', \bz-\bz'\right\rangle+\frac{1}{\gamma}\|\bz-\bz'\|^2\geq\left(\frac{1}{\gamma}-\rho\right)\|\bz-\bz'\|^2,
	\end{eqnarray*}
	\normalsize
	where the second inequality is because of the $\rho$-weakly monotonicity of $F$. Since $F_{\bw}^\gamma(\bz)$ consists of all vectors like $\bxi+\frac{1}{\gamma}(\bz-\bw)$ with $\bxi\in F(\bz)$, we conclude that  $F_{\bw}^\gamma$ is $(\frac{1}{\gamma}-\rho)$-strongly monotone.
\end{proof}

\begin{proof}[Proof of Lemma~\ref{lem:zbw2}]
	Since $\bar\bw$ is the solution to $\text{SVI}(F_{\bw}^\gamma,\Z)$ and $F_{\w}$ is strongly monotone, then $\bar\w$ is the min-max saddle-point of  the convex-concave problem $\min_{\x\in\X}\max_{\y\in\Y}f(\x, \y) + \frac{1}{2\gamma}\|\x- \u\|^2 - \frac{1}{2\gamma}\|\y- \v\|^2$~\citep{Nemirovski2015InformationbasedCO}. Denote by $f_\w(\x, \y) = f(\x, \y) + \frac{1}{2\gamma}\|\x- \u\|^2 - \frac{1}{2\gamma}\|\y- \v\|^2$.  Then 
	$$
	\mathbf{0}\in\partial_x[f_{\bw}(\bar\bu,\bar\bv)+\mathbf{1}_{\X}(\bar\bu)],\quad \mathbf{0}\in\partial_y[-f_{\bw}(\bar\bu,\bar\bv)+\mathbf{1}_{\y}(\bar\bv)] 
	$$ 
	which means there exist $\bar\bxi_{x}\in\partial_x[f(\bar\bu,\bar\bv)+\mathbf{1}_{\X}(\bar\bu)]$ and $-\bar\bxi_{y}\in\partial_y[-f(\bar\bu,\bar\bv)+\mathbf{1}_{\y}(\bar\bv)] $ such that
	$$
	\|\bar\bxi\|=\frac{1}{\gamma}\|\bar\bw-\bw\|,\quad \bar\bxi=(\bar\bxi_{x},-\bar\bxi_{y}),
	$$
	which completes the proof. 
	%This means $\widehat\bu$ is an $\frac{\epsilon}{\gamma}$-stationary point of the problem $\min_{\bx\in\X}f(\x, \widehat\bv)$ and $\widehat\bv$ is an $\frac{\epsilon}{\gamma}$-stationary point of the problem $\max_{\by\in\Y}f(\widehat\bu, \by)$, justifying the quality of $(\widehat\bu,\widehat\bv)$ as a solution to the non-convex-concave problem \eqref{eqn:SPP}.
	%
	%	
	%	\normalsize
	%	for any $\bz\in\Z$. The conclusion is proved by reorganizing terms and using the fact that $\|\bz-\widehat\bw\|\leq  D$.
	%	%$\|\bz-\widehat\bw\|\leq \sqrt{2V(\bz,\widehat\bw)} \leq D$
\end{proof}

%\subsection{Proof of Lemma~\ref{lem:prom}}
\begin{proof}[Proof of Lemma~\ref{lem:prom}]
	%The main idea of this proof is originally from~\cite{DBLP:journals/coap/DangL15}.  
	We use the following notations $\z^{(t)}=\zh$, $\z^{(t+1)} = \bar\z  = \text{Proj}_{\Z}(\z^{(t)} - \eta F(\z^{(t)}))$.  For any $\z$, we define
	\begin{align*}
	\phi_\eta(\z) := F(\z) - F(\z_+)+ \frac{1}{\eta}[\z_+ - \z],
	\end{align*}
	where $\z_+ = \text{Proj}_{\Z}(\z - \eta F_k(\z))$. Then we have $\|\phi_\eta(\z)\|\leq (L+\eta^{-1})\|\z - \z_+\| $ and thus
	\begin{align}
	\nonumber
	\max_{\z\in\Z}F(\z^{(t+1)})^{\top}(\z^{(t+1)} - \z) &=  \max_{\z\in\Z}(F(\z^{(t+1)}) + \phi_\eta(\z^{(t)}) - \phi_\eta(\z^{(t)}))^{\top}(\z^{(t+1)} - \z)\\\label{eqnkey20}
	&\leq \max_{\z\in\Z} -\phi_\eta(\z^{(t)})^{\top}(\z^{(t)} - \z)\leq D(L+\eta^{-1})\|\z^{(t)} - \z^{(t+1)}\|,
	\end{align}
	where the first inequality uses the optimality condition of $\z^{(t+1)} = \arg\min_{\z\in\Z}\|\z - (\z^{(t)} - \eta F(\z^{(t)}))\|^2$ that says
	$$
	\left[(\z^{(t+1)}  - \z^{(t)})/\eta +  F(\z^{(t)}))\right]^\top(\z-\z^{(t+1)})=\left[F(\z^{(t+1)}) +\phi_\eta(\z^{(t)})\right]^\top(\z-\z^{(t+1)})\geq0,\quad\forall\z\in\Z.
	$$ 
	
	To continue, we define $\y^{(t+1)} = \arg\min_{\z\in\Z}\|\z - (\z^{(t)} - \eta F(\z^{(t+1)}))\|^2$. According to Lemma 5 in~\citep{DBLP:journals/coap/DangL15}, we have
	\begin{align}%\label{eqnkey}
	(1 - L^2\eta^2)\frac{\|\z^{(t)}- \z^{(t+1)}\|^2}{2}\leq \frac{\|\z^{(t)} - \w_*\|^2 - \|\y^{(t+1)} - \w_*\|^2 }{2}.
	\end{align}
	Plugging the value of $\eta$ gives
	\begin{align*}
	\|\zh- \bar\z\|\leq \sqrt{2}\|\zh- \w_*\|.
	\end{align*}
	Combining this inequality with~(\ref{eqnkey20}) we have
	%The following similar analysis as in the proof of Proposition~\ref{prop:gd}, we can show that 
	%	\begin{align*}
	%	\max_{\z\in\Z}F(\bar\z)^{\top}(\bar \z - \z) &\leq D(L+\eta^{-1})\|\zh- \bar\z\| 
	%	\end{align*}
	%	and
	%	\begin{align*}
	%	(1 - L^2\eta^2)\frac{\|\bar\z- \zh \|^2}{2}\leq \frac{\|\zh - \w_*\|^2 }{2}
	%	\end{align*}
	%	By plugging the value of $\eta$,
	%\begin{align*}
	%\|\zh- \bar\z\|\leq \sqrt{2}\|\zh- \w_*\|
	%\end{align*}
	%Thus, 
	\begin{align*}
	\max_{\z\in\Z}F(\bar\z)^{\top}(\bar\z - \z) &\leq DL(1+ \sqrt{2})\sqrt{2}\|\zh - \w_*\|
	\end{align*}
\end{proof}

\section{Weakly-Convex-Weak-Concave Examples}
\label{sec:examples}
%\paragraph{Examples and Discussions.} Finally, 
In this section, we present some examples of the min-max problem whose objective function is weakly-convex and weakly-concave.
%, and also provide more discussion on the key Assumption~\ref{assume:MVIexist} (iii). 

{\bf Example 1: } When  $f(\x, \y)$ is $L$-smooth function in terms of $\x$ when fixing $\y$ and $L$-smooth function in terms of $\y$ when fixing $\x$, it is $L$ weakly-convex-weakly-concave. This kind of problems can be found in training GAN~\cite{pmlr-v70-arjovsky17a} and reinforcement learning~\cite{DBLP:conf/icml/DaiS0XHLCS18}. Given the smoothness, we have
\begin{align*}
f(\x, \y)\geq f(\x', \y) + \nabla_x f(\x', \y)^{\top}(\x  -\x') - \frac{L}{2}\|\x - \x'\|^2, \forall \y\in\Y
\end{align*}
which gives us
\begin{align*}
f(\x, \y) + \frac{L}{2}\|\x\|^2\geq f(\x', \y) + \frac{L}{2}\|\x'\|^2+ (\nabla_x f(\x', \y)+L\x')^{\top}(\x  -\x'), \forall \y\in\Y,
\end{align*}
which means $f(\x, \y) + \frac{L}{2}\|\x\|^2$ is convex in terms of $\x$ for any fixed $\y\in\Y$.  Similarly, we can prove $f(\x, \y) -\frac{L}{2}\|\y\|^2$ is a concave function in terms of $\y$ for any $\x\in\X$.

{\bf Example 2: } Let us consider  $f(\x, \y)  = \x^{\top}A\y + \phi(g(\x)) -   \psi(h(\y))$ where  $\phi(\cdot)$ and $\psi(\cdot)$ are  Lipschitz continuous convex functions, $g(\x)$ and $h(\y)$ are smooth mappings. Following~\citep{dmitriy2017}, we can prove that $f(\x, \y)$ is weakly convex in terms of $\x$ when fixing $\x$, and weakly concave in terms of $\y$ when fixing $\y$.  More specifically, 
by the convexity of $\phi(\cdot)$, we have
\begin{align*}
\phi(g(\x))\geq \phi(g(\x')) + \nabla\phi(g(\x'))^{\top}(g(\x) - g(\x'))
\end{align*}
By the smoothness of $g(\x)$, we know there exists $L>0$ such that 
\begin{align*}
\|g(\x) - g(\x') - \nabla g(\x')(\x - \x')\|\leq \frac{L}{2}\|\x - \x'\|^2
\end{align*}
Combining the above two inequalities we have
\begin{align*}
\phi(g(\x))\geq \phi(g(\x')) + \nabla\phi(g(\x'))^{\top} \nabla g(\x')(\x - \x') - \frac{L\|\nabla\phi(g(\x'))\|}{2}\|\x - \x'\|^2
\end{align*}
Since $\phi$ is Lipschitz continuous, there exists $M>0$ such that $\|\nabla\phi(g(\x'))\|\leq M$. As a result, 
\begin{align*}
\phi(g(\x))\geq \phi(g(\x')) + \nabla\phi(g(\x'))^{\top} \nabla g(\x')(\x - \x') - \frac{LM}{2}\|\x - \x'\|^2.
\end{align*}
This proves the weak convexity of $\phi(g(\x))$. Similarly we can prove the weak convexity of $\psi(h(\y))$. Since $\x^{\top}A\y$ is convex in terms of $\x$ when fixing $\y$ and concave in terms of $\y$ when fixing $\x$. Thus, we have weak-convexity and weak-concavity of $f(\x, \y)$. 

It is notable that $f(\x, \y)$ is not necessarily smooth. 

{\bf Example 3: } Let us consider  $f(\x, \y)  = \phi(g(\x) -  h(\y))$ where  $\phi(\cdot)$ is a non-decreasing smooth function, $g(\x)$ and $h(\y)$ are Lipschitz continuous convex functions. Following~\cite{DBLP:journals/corr/abs-1805-07880}, it can be proved that $f(\x, \y)$ is weakly convex in terms of $\x$ when fixing $\y$, and weakly concave in terms of $\y$ when fixing $\x$. The following inequalities hold for any  $\x, \x'\in\X, \y, \y'\in\Y$.  In fact, 
by the smoothness of $\phi$, there exists $L>0$ such that 
\begin{align*}
\phi(g(\x) - h(\y))\geq \phi(g(\x') -  h(\y)) + \phi'(g(\x') -  h(\y))(g(\x) - g(\x')) - \frac{L}{2}|g(\x) - g(\x')|^2,% \forall \x, \x'\in\X, \y\in\Y
\end{align*} 
Since $g(\x)$ is convex and Lipschitz continuous, we have
\begin{align*}
g(\x)- g(\x')\geq \nabla g(\x')^{\top}(\x -\x'), \quad |g(\x) - g(\x')|\leq G_x\|\x - \x'\|
\end{align*}
Noting that $\phi$ is non-decreasing function with $\phi'(\cdot)\geq 0$, then we have
\begin{align*}
\phi(g(\x) - h(\y))\geq \phi(g(\x') - h(\y)) + \phi'(g(\x') - h(\y))\nabla g(\x')^{\top}(\x -\x') - \frac{LG_x^2}{2}\|\x - \x'\|^2,% \forall \x, \x'\in\X, \y\in\Y,
\end{align*} 
which implies weak-convexity in terms of $\x$ when fixing $\y$.  Similarly, 
\begin{align*}
\phi(g(\x) - h(\y))\leq \phi(g(\x) -  h(\y')) -  \phi'(g(\x) -  h(\y'))(h(\y) - h(\y')) +  \frac{L}{2}|h(\y) - h(\y')|^2, %\forall \x\in\X, \y, \y'\in\Y
\end{align*} 
Since $h(\x)$ is convex and Lipschitz continuous, we have
\begin{align*}
h(\y)- h(\y')\geq \nabla h(\y')^{\top}(\y -\y'), \quad |h(\x) - h(\x')|\leq G_y\|\y - \y'\|
\end{align*}
Then we have
\begin{align*}
\phi(g(\x) - h(\y))\leq \phi(g(\x) -  h(\y')) -  \phi'(g(\x) -  h(\y'))\nabla h(\y')^{\top}(\y -\y') +  \frac{LG_y^2}{2}\|\y - \y'\|^2, %\forall \x\in\X, \y, \y'\in\Y,
\end{align*} 
which implies weak-concavity in terms of $\y$ when fixing $\x$.  

It also is notable that $f(\x, \y)$ is not necessarily smooth. This problem can be found in robust statistics~\citep{audibert2011}. 

{\bf Example 4 (An Example Satisfying MVI Condition): }
We consider
$f(\x,\y) = \frac{1}{2}\x^\top A\x+\x^\top \y-\frac{1}{2}\y^\top A\y$,  $\X=\{(x_1,x_2)\in\R^2:0\leq x_2\leq x_1\leq 1\}$, $\Y=\{(y_1,y_2)\in\R^2:0\leq y_2\leq y_1\leq 1\}$,  where  $A=\text{diag}(1,-\rho)$, $0<\rho<1$. It is not difficult to show that Assumption \ref{ass:3} holds, i.e., $D=\sqrt{2}$, the mapping $F(\x,\y)=(\partial_{\x}f(\x,y), -\partial_{\y}f(\x,\y))^\top $ is $\rho$-weakly monotone and the corresponding MVI problem has a solution $(0,0,0,0)$. In addition, $F(\x,\y)$ is Lipschitz continuous with modulus $L=1$. 

\section{Proof of Theorem~\ref{thm:meta}}
\label{sec:proofthmmeta}

%\begin{proof}[Proof of Theorem~\ref{thm:meta}]
Let $\E_k$ be the conditional expectation conditioning on all the stochastic events until $\bz_k$ is generated.
Let $\bar\z_k$ be the unique solution of $\text{SVI}(F_k,\Z)$ where $F_k$ is defined in Algorithm~\ref{alg:meta}. This means 
\begin{eqnarray}\label{eq:ineq1}
\exists \bar\bxi_k\in F_k(\bar\z_k) \text{ s.t. } \bar\bxi_k^{\top}(\z - \bar\z_k)\geq 0,~\forall \bz\in\Z.
\end{eqnarray}
By the assumption on $\text{ApproxSVI}(F_k, \Z, \z_k, \eta_k , T_k)$ and the fact that $\text{ApproxSVI}$ only depends on all the previous stochastic events thought $\bz_k$, we have
\begin{eqnarray}\label{eq:ineq2}
\exists \bxi_{k+1}\in F_k(\z_{k+1}) \text{ s.t. }		\E_k[\bxi_{k+1}^{\top}(\z_{k+1} - \z)]\leq \frac{c}{k+1},~\forall \bz\in\Z.
\end{eqnarray}
%where 
%			\begin{eqnarray*}
%		\frac{c}{k+1}=\left\{
%		\begin{array}{ll}\frac{D^2}{2\eta_k T_k} + \frac{\eta_k G^2}{2}&\text{if }\bz_{k+1}=\text{SG}(F_k,\Z,\bz_k,\eta_k,T_k)\\\frac{D^2}{\eta_k T_k}&\text{if }\bz_{k+1}=\text{EG}(F_k,\Z,\bz_k,\eta_k,T_k)\end{array}
%		\right.
%	\end{eqnarray*}
By the $(\gamma^{-1} - \rho)$-strong monotonicity of $F_k$, we have
%\small
\begin{eqnarray}\label{eq:ineq3}
(\gamma^{-1} - \rho)\E_k\|\z_{k+1} -\bar \z_k\|^2\leq \E_k[( \bxi_{k+1} -  \bar\bxi_k)^{\top}(\z_{k+1} - \bar\z_k)]\leq \frac{c}{k+1}
\end{eqnarray}
%\normalsize
where the second inequality is obtained using \eqref{eq:ineq1} with $\bz=\z_{k+1}$ and using \eqref{eq:ineq2} with $\bz=\bar\z_k$.

Let $\zu$ be a solution to $\text{MVI}(F, \Z)$, meaning that  $\bxi^{\top}(\z - \zu)\geq 0$ for any $\bz\in\Z$ and any $\bxi\in F(\bz)$. Note that such a solution exists by Assumption~\ref{ass:3}.
According to the definition of $F_k(\z_{k+1})$ and the fact that $\bxi_{k+1}\in F_k(\bz_{k+1})$, we have 
$
\bxi_{k+1}-\gamma^{-1}(\bz_{k+1}-\bz_k)\in F(\bz_{k+1})
$ 
so that 
$$
(\bxi_{k+1}-\gamma^{-1}(\bz_{k+1}-\bz_k))^{\top}(\z_{k+1} - \zu)\geq 0
$$
by the definition of $\zu$. This inequality and \eqref{eq:ineq2} with $\bz=\zu$ together imply
\begin{eqnarray*}
	\E_k[(\z_k - \z_{k+1})^{\top}(\z_{k+1} - \zu)] \geq \gamma 	\E_k[\bxi_{k+1}^{\top}(\zu - \z_{k+1})]\geq -  \frac{\gamma c}{k+1}.
\end{eqnarray*}
As a result, we have
\small
\begin{eqnarray*}
	&&\E_k\|\z_{k} - \zu\|^2 \\
	&=& \E_k\|\z_k - \z_{k+1} - \z_{k+1} - \zu\|^2\\
	&  =&\E_k \|\z_k - \z_{k+1}\|^2 +\E_k \|\z_{k+1} - \zu\|^2 + 2\E_k[(\z_k - \z_{k+1})^{\top}(\z_{k+1} - \zu)]\\
	& \geq&\E_k\|\z_k - \z_{k+1}\|^2+ \E_k\|\z_{k+1} - \zu\|^2  - \frac{2\gamma c}{k+1}\\
	& =& \E_k\|\z_k - \bar\z_k \|^2 +\E_k \|\z_{k+1} - \bar\z_k\|^2 + 2\E_k[(\z_k - \bar\z_k)^{\top}(\bar\z_k - \z_{k+1} )]+\E_k \|\z_{k+1} - \zu\|^2  - \frac{2\gamma c}{k+1}\\
	& \geq& \E_k \|\z_k - \bar\z_k \|^2 + \E_k\|\z_{k+1} - \bar\z_k\|^2 - a\E_k\|\z_k -\bar\z_k\|^2 - a^{-1}\E_k\|\bar\z_k - \z_{k+1}\|^2+ \E_k\|\z_{k+1} - \zu\|^2  - \frac{2\gamma c}{k+1},
\end{eqnarray*}
\normalsize
where the last inequality is by Young's inequality and a constant $a\in(0,1)$.
Rearranging the  inequality above gives
\begin{eqnarray*}
	(1-a) \E_k\|\z_k - \bar\z_k \|^2&\leq&( \E_k\|\z_k - \zu\|^2 -  \E_k\|\z_{k+1} - \zu\|^2) +  (a^{-1}-1) \E_k\|\bar\z_k - \z_{k+1}\|^2 + \frac{2\gamma c}{k+1}\\
	& \leq&( \E_k\|\z_k - \zu\|^2 -  \E_k\|\z_{k+1} - \zu\|^2)  +\left( \frac{a^{-1}-1}{\gamma^{-1} - \rho}+2\gamma\right)\frac{c}{k+1},
\end{eqnarray*}
where the second inequality holds because of \eqref{eq:ineq3}.
Let $\theta_{-1}=0$. 
Multiplying  both sides of the inequality above by $\theta_k$, taking expectation over all random events, and taking summation over $k=0,1,\dots,K-1$, we have 
\small
\begin{eqnarray}\nonumber
&&\sum_{k=0}^{K-1}(1-a)\theta_k \E\|\z_k - \bar\z_k \|^2\\\nonumber
&\leq&\sum_{k=0}^{K-1}\theta_k(\E\|\z_k - \zu\|^2 - \E\|\z_{k+1} - \zu\|^2) + \left( \frac{a^{-1}-1}{\gamma^{-1} - \rho}+2\gamma\right)\sum_{k=0}^{K-1}\frac{c \theta_k}{k+1}\\\nonumber
&=&\sum_{k=0}^{K-1}(\theta_{k-1}\E\|\z_k - \zu\|^2 - \theta_k\E\|\z_{k+1} - \zu\|^2) + \sum_{k=0}^{K-1}(\theta_{k} - \theta_{k-1})\E\|\z_k - \zu\|^2
+\left( \frac{a^{-1}-1}{\gamma^{-1} - \rho}+2\gamma\right)\sum_{k=0}^{K-1}\frac{c \theta_k}{k+1}\\\nonumber
&=&\theta_{-1}\E \|\z_0 - \zu\|^2 - \theta_{K-1}\E\|\z_K - \zu\|^2 + (\theta_{K-1} - \theta_{-1})D^2+\left( \frac{a^{-1}-1}{\gamma^{-1} - \rho}+2\gamma\right)\sum_{k=0}^{K-1}\frac{c \theta_k}{k+1}\\\label{eq:ineq4}
&\leq &\theta_{K-1}D^2+\left( \frac{a^{-1}-1}{\gamma^{-1} - \rho}+2\gamma\right)\sum_{k=0}^{K-1}\frac{c \theta_k}{k+1}
\end{eqnarray}
\normalsize
Given that $\gamma=1/(2\rho)$ and the definition of $\tau$, by setting $a=1/2$ and dividing \eqref{eq:ineq4} by $\frac{1}{2}\sum_{k=0}^{K-1}\theta_k$, we have
\begin{eqnarray}\label{eq:ineq5}
\E[ \|\z_\tau - \bar\z_\tau \|^2]\leq\frac{2D^2\theta_{K-1}}{\sum_{k=0}^{K-1}\theta_k}+ \frac{4\sum_{k=0}^{K-1}\frac{c\theta_k}{k+1}}{\rho\sum_{k=0}^{K-1}\theta_k}.
%\leq \frac{2D^2\theta_{K-1}}{\sum_{k=0}^{K-1}\theta_k}+ \frac{4c\sum_{k=0}^{K-1}\theta_k/(k+1)}{\rho\sum_{k=0}^{K-1}\theta_k}.%\leq \frac{2D^2\sqrt{K}}{(2/3)K^{3/2}} + \frac{4\sum_{k=0}^{K-1}\sqrt{k+1}\frac{c}{k+1}}{(2/3)\rho K^{3/2}}.
\end{eqnarray}
Since $\alpha>0$, standard calculus yields
\begin{align*}
\sum_{k=1}^Kk^{\alpha}&\geq \int_0^{K}x^\alpha d x= \frac{1}{\alpha+1}K^{\alpha+1}\\
%\frac{S^\alpha}{\sum_{s=1}^S s^\alpha} &\leq \frac{S^\alpha}{\int_0^{S}x^\alpha d x}= \frac{(\alpha+1)S^\alpha}{S^{\alpha+1}}\leq \frac{\alpha+1}{S}\\
\sum_{k=1}^Kk^{\alpha-1}&\leq KK^{\alpha-1} = K^\alpha, \quad \text{if }\alpha\geq 1\\
\sum_{k=1}^Kk^{\alpha-1}&\leq \int_{0}^K x^{\alpha-1}d x = \frac{ K^\alpha}{\alpha}, \quad \text{if } 0<\alpha<1.
\end{align*}
Applying the fact that $\theta_k = (k+1)^\alpha$ and the above inequalities into~(\ref{eq:ineq5}), we have
\begin{eqnarray*}\label{eq:final}
	\E[ \|\z_\tau - \bar\z_\tau \|^2]\leq \frac{2D^2(\alpha+1)}{K}+ \frac{4c(\alpha+1)}{K \rho \alpha^{\mathbf{1}_{\alpha<1}}}.%\leq \frac{2D^2\sqrt{K}}{(2/3)K^{3/2}} + \frac{4\sum_{k=0}^{K-1}\sqrt{k+1}\frac{c}{k+1}}{(2/3)\rho K^{3/2}}.
\end{eqnarray*}

%		Next, we consider the two updates for $\z_k$ in Algorithm~\ref{alg:meta}. Suppose $\bz_{k+1}=\text{SG}(F_k,\Z,\bz_k,\eta_k,T_k)$. We have $\frac{c}{k+1} = \frac{D^2}{2\eta_k T_k}+ \frac{\eta_kG^2}{2}$. By setting $\eta_k = \frac{D}{G(k+1)}$ and $T_k = (k+1)^2$, we have
%		$\frac{c}{k+1}=\frac{DG}{k+1}$ so that \eqref{eq:ineq5} implies
%		\begin{eqnarray*}
%			\E[ \|\z_\tau - \bar\z_\tau \|^2]\leq\frac{3D^2}{K}+ \frac{6\sum_{k=0}^{K-1}\frac{DG}{\sqrt{k+1}}}{\rho K^{3/2}}
%			\leq\frac{3D^2}{K}+ \frac{12DG\sqrt{K}}{\rho K^{3/2}}=\left(3D^2+\frac{12DG}{\rho}\right)\frac{1}{K}
%		\end{eqnarray*}
%		Hence, in order to ensure $\E[ \|\z_\tau - \bar\z_\tau \|^2]\leq \epsilon^2$, we need $K=O(1/\epsilon^2)$ and total iteration complexity is $\sum_{k=0}^{K-1}T_k = \sum_{k=0}^{K-1}(k+1)^2 = O(K^3) = O(1/\epsilon^6)$. 
%		
%		
%		Suppose $\bz_{k+1}=\text{SG}(F_k,\Z,\bz_k,\eta_k,T_k)$, we have $\frac{c}{k+1} = \frac{D^2}{\eta_k T_k}$. By setting $\eta_k=1/L$ and $T_k = k+1$, we have $\frac{c}{k+1}=\frac{LD^2}{k+1}$ so that \eqref{eq:ineq5} implies
%		\begin{eqnarray*}
%			\E[ \|\z_\tau - \bar\z_\tau \|^2]\leq\frac{3D^2}{K}+ \frac{6\sum_{k=0}^{K-1}\frac{LD^2}{\sqrt{k+1}}}{\rho K^{3/2}}
%			\leq\frac{3D^2}{K}+ \frac{12LD^2\sqrt{K}}{\rho K^{3/2}}=\left(3D^2+\frac{12LD^2}{\rho}\right)\frac{1}{K}
%		\end{eqnarray*}
%		Hence, in order to ensure $\E[ \|\z_\tau - \bar\z_\tau \|^2]\leq \epsilon^2$, we need $KK=O(1/\epsilon^2)$ and total iteration complexity is $\sum_{k=0}^{K-1}T_k = \sum_{k=0}^{K-1}(k+1) = O(1/\epsilon^4)$. 
%		
%		
%\end{proof}
\section{Proof of Proposition~\ref{prop:alg1}}
\label{sec:sgdproof}
The proof is following  the standard analysis of stochastic subgradient method. By the updating steps in Algorithm~\ref{alg:SG},  for any $\z\in\Z$, we have
\begin{align*}
\bzeta(\z^{(t)})^{\top}(\z^{(t)} - \z)\leq \frac{\|\z^{(t)} - \z\|^2 -  \|\z^{(t+1)} - \z\|^2}{2\eta}    + \frac{\eta\|\bzeta(\z^{(t)}\|^2}{2}
\end{align*}
Adding the inequalities  for $t = 0,1,2,..., T-1$ leads to
\begin{align*}
\sum^{T-1}_{t=0}\bzeta(\z^{(t)})^\top (\z^{(t)} - \z) \leq \dfrac{1}{2\eta} || \z^{(0)} - \z||^2 + \dfrac{\eta}{2}\sum^{T-1}_{t=0}   || \bzeta(\z^{(t)}) ||^2
\end{align*}
Let $\tau $ be a uniform random index from $\{0,1,2,...,T-1\}$. Using the previous inequality, we can show that
\begin{align*}
\mathbb{E}[(\bxi^{(\tau)})^\top (\z^{(\tau)} - \z)] &\leq \dfrac{|| \z^{(0)} - \z||^2}{2 \eta T} + \dfrac{\eta}{2} \mathbb{E}[  || \bzeta(\z^{(\tau)}) ||^2 ] \\ 
&\leq \dfrac{D^2}{2\eta T} + \dfrac{\eta G_k^2}{2}.
\end{align*}
\iffalse
If $F_k$ is $\mu$-strongly monotone, we have 
\begin{align*}
\mu ||\z^{(t)} - \w_*  ||^2 \leq (\bxi^{(t)} - \bxi_*)^\top ( \z^{(t)} - \w_* ) \leq (\bxi^{(t)} )^\top ( \z^{(t)} - \w_* )
\end{align*}
where $\w_*$ denotes a solution to $\text{SVI}(F_k, \Z)$ and $\bxi_{*}\in F(\w_*)$. Taking the expectation of this inequality yields the second conclusion. 
\fi 

\section{Proof of Proposition~\ref{prop:gd}: Part I for GD}
\label{sec:gdproof}
We first show that $\|\z^{(t)} - \w_*\|$ converges to zero linearly where $\w_*$ is a solution to $\text{SVI}(F_k, \Z)$. 
The proof is standard and can be found in~\cite{1078-0947Nesterov}. We include it here only for the completeness.
The definition of $\w_*$ implies $\w_* = \text{Proj}_{\Z}(\w_* - \eta F_k(\w_*))$. Using the fact that $\eta = \mu/(2L^2)$ and the non-expansion property of $\text{Proj}_{\Z}(\cdot)$, we have 
\begin{align*}
\|\z^{(t+1)} - \w_*\|^2 &\leq \|\z^{(t)} - \eta F_k(\z^{(t)}) - \w_* + \eta F_k(\w_*)\|^2\\
&=\|\z^{(t)}  - \w_*\|^2 -2\eta ( F_k(\z^{(t)}) - F_k(\w_*))^{\top}(\z^{(t)} - \w_*) + \eta^2 \|F_k(\z^{(t)}) - F_k(\w_*)\|^2\\
&\leq (1 - 2\eta\mu + \eta^2 L^2)\|\z^{(t)} -\w_* \|^2 = (1 - 3\mu^2/(4L^2))\|\z^{(t)} -\w_* \|^2\\
&\leq\exp(-\frac{3}{4\beta^2}) \|\z^{(t)} -\w_* \|^2
\end{align*}
where the second inequality holds because of the $\mu$-strong monotonicity and $L$-Lipschitz continuity of $F_k$. Applying this inequality for $t=0,1,\dots,$ gives 
\begin{align}
\label{eqnkey0}
\|\z^{(t+1)} - \w_*\|^2\leq \exp(-\frac{3(t+1)}{4\beta^2}) \|\z^{(0)} -\w_* \|^2.
%\phi_\eta(\z) = F_k(\z) - F_k(\z_+)+ \frac{1}{\eta}[\z_+ - \z],
\end{align}
Next, we prove that $\max_{\z\in\Z}F_k(\z^{(t+1)})^{\top}(\z^{(t+1)} - \z)\leq \|\z^{(t)} - \w_*\|^2$. The main idea of this proof is originally 
from~\cite{DBLP:journals/coap/DangL15}. We need to introduce
\begin{align*}
\phi_\eta(\z) := F_k(\z) - F_k(\z_+)+ \frac{1}{\eta}[\z_+ - \z],
\end{align*}
where $\z_+ = \text{Proj}_{\Z}(\z - \eta F_k(\z))$. Then we have $\|\phi_\eta(\z)\|\leq (L+\eta^{-1})\|\z - \z_+\| $ and thus
\begin{align}
\nonumber
\max_{\z\in\Z}F_k(\z^{(t+1)})^{\top}(\z^{(t+1)} - \z) &=  \max_{\z\in\Z}(F_k(\z^{(t+1)}) + \phi_\eta(\z^{(t)}) - \phi_\eta(\z^{(t)}))^{\top}(\z^{(t+1)} - \z)\\\label{eqnkey2}
&\leq \max_{\z\in\Z} -\phi_\eta(\z^{(t)})^{\top}(\z^{(t)} - \z)\leq D(L+\eta^{-1})\|\z^{(t)} - \z^{(t+1)}\|,
\end{align}
where the first inequality uses the optimality condition of $\z^{(t+1)} = \arg\min_{\z\in\Z}\|\z - (\z^{(t)} - \eta F_k(\z^{(t)}))\|^2$ that says
$$
\left[(\z^{(t+1)}  - \z^{(t)})/\eta +  F_k(\z^{(t)}))\right]^\top(\z-\z^{(t+1)})=\left[F_k(\z^{(t+1)}) +\phi_\eta(\z^{(t)})\right]^\top(\z-\z^{(t+1)})\geq0,\quad\forall\z\in\Z.
$$ 

To continue, we define $\y^{(t+1)} = \arg\min_{\z\in\Z}\|\z - (\z^{(t)} - \eta F_k(\z^{(t+1)}))\|^2$. According to Lemma 5 in~\cite{DBLP:journals/coap/DangL15}, we have
\begin{align}\label{eqnkey}
(1 - L^2\eta^2)\frac{\|\z^{(t)}- \z^{(t+1)}\|^2}{2}\leq \frac{\|\z^{(t)} - \w_*\|^2 - \|\y^{(t+1)} - \w_*\|^2 }{2}.
\end{align}
Plugging the value of $\eta$ into \eqref{eqnkey} gives
\begin{align}
\label{eqnkey1}
\|\z^{(t)}- \z^{(t+1)}\|\leq \sqrt{\frac{1}{(1 - 0.25\beta^{-2})}}\|\z^{(t)} - \w_*\|
\end{align}
Finally, using \eqref{eqnkey0}, \eqref{eqnkey2} and \eqref{eqnkey1} together, we can show that
\begin{align*}
\max_{\z\in\Z}F_k(\z^{(t+1)})^{\top}(\z^{(t+1)} - \z) &\leq DL(1+ 2\beta)\sqrt{\frac{1}{(1 - 0.25\beta^{-2})}}\|\z^{(t)} - \w_*\|\\
&\leq DL(1+ 2\beta)\sqrt{\frac{\beta^2}{\beta^2 - 0.25}}\exp(-\frac{3t}{8\beta^2}) \|\z^{(0)} -\w_* \|\\
&\leq DL(1+2\beta)\sqrt{\frac{\beta^2}{\beta^2 - 0.25\beta^2}}\exp(-\frac{3t}{8\beta^2}) \|\z^{(0)} -\w_* \|\\
&\leq 4D^2L\beta\exp(-\frac{t}{4\beta^2}), 
\end{align*}
where we use the fact $\beta\geq 1$ in the third and the fourth inequalities. The first conclusion is then proved by setting $t=T-1$.

Since $F_k$ is $\mu$-strongly monotone, we have 
\begin{align*}
\mu ||\z^{(T)} - \w_*  ||^2 \leq (F_k(\z^{(T)}) - F_k(\w_* ))^\top ( \z^{(T)} - \w_* ) \leq (F(\z^{(T)}))^\top ( \z^{(T)} - \w_* )
\end{align*}
where $\w_*$ denotes a solution to $\text{SVI}(F_k, \Z)$, which yields the second conclusion. 

\section{Proof of Proposition~\ref{prop:gd}: Part II for EG}
It is worth mentioning that the linear convergence of the extragradient for strongly monotone VI  in terms of distance to the optimal solution is also proved in~\citep{sim18ganvi}. For completeness, we present a proof here.  We can use the following lemma to prove the linear convergence. 
\begin{lemma}[\textbf{Lemma 3.1}~\cite{nemirovski-2005-prox}]\label{lem:6}
	Let $\omega(\z)$ be a  $\alpha$-strongly convex function with respect to the norm $\|\cdot\|$, whose dual norm is denoted by $\|\cdot\|_*$,  and $D(\x,\z) = \omega(\x)- (\omega(\z) + (\x-\z)^{\top}\omega'(\z))$ be the Bregman distance induced by function $\omega(\x)$. Let $Z$ be a convex compact set, and $U\subseteq Z$ be convex and closed.  Let $\z\in Z$, $\gamma>0$, Consider the points,
	\begin{align}
	\x &= \arg\min_{\u\in U} \gamma\u^{\top}\xi + D(\u, \z)\label{eqn:project1},\\
	\z_+&=\arg\min_{\u\in U} \gamma\u^{\top}\zeta + D(\u,\z),\label{eqn:project2}
	\end{align}
	then for any $\u\in U$, we have
	\begin{align}\label{eqn:ineq}
	\gamma\zeta^{\top}(\x-\u)\leq  D(\u,\z) - D(\u, \z_+) + \frac{\gamma^2}{\alpha}\|\xi-\zeta\|_*^2 - \frac{\alpha}{2}[\|\x-\z\|^2 + \|\x-\z_+\|^2].
	\end{align}
\end{lemma}

\begin{proof}
	Following the Lemma~\ref{lem:6}, we can easily have the following equality for the extragrdient method 
	\begin{align*}
	2\eta F(\z^{(t)})^{\top}(\z^{(t)}-\w_*)&\leq \|\w_* - \w^{(t)}\|^2 - \|\w_* - \w^{(t+1)}\|^2 + 2\eta^2\|F(\z^{(t)}- F(\w^{(t)})\|^2\\
	&- [\|\z -\w^{(t+1)}\|^2 + \|\z-\w^{(t)}\|^2].
	\end{align*}
	%By taking summation over $t=0,\ldots, T-1$ and noting the Lipschitz continuity of  $F(\z)$, we can finish the proof. 
	By the strong monotonicity, we have
	\begin{align*}
	\frac{1}{2}\mu\|\w^{(t)}-\w_*\|^2-  \mu\|\z^{(t)}-\w^{(t)}\|^2&\leq    \mu\|\z^{(t)}-\w_*\|^2\leq (F(\z^{(t)}) - F(\w_*))^{\top}(\z^{(t)}-\w_*)\\
	&\leq F(\z^{(t)})^{\top}(\z^{(t)}-\w_*).
	\end{align*}
	Combining the above inequalities, we have
	
	\begin{align*}
	\eta\mu \|\w^{(t)}-\w_*\|^2-  2\eta\mu\|\z^{(t)}-\w^{(t)}\|^2&\leq \|\w_* - \w^{(t)}\|^2 - \|\w_* - \w^{(t+1)}\|^2 + 2\eta^2\|F(\z^{(t)}- F(\w^{(t)})\|^2\\
	&- [\|\z -\w^{(t+1)}\|^2 + \|\z-\w^{(t)}\|^2].
	\end{align*}
	Reorganizing the terms we have
	\begin{align*}
	\|\w_* - \w^{(t+1)}\|^2&\leq (1-\eta\mu) \|\w^{(t)}-\w_*\|^2 +   2\eta\mu\|\z^{(t)}-\w^{(t)}\|^2  + 2\eta^2\|F(\z^{(t)}- F(\w^{(t)})\|^2\\
	&- [\|\z -\w^{(t+1)}\|^2 + \|\z-\w^{(t)}\|^2]\\
	&\leq (1-\eta\mu) \|\w^{(t)}-\w_*\|^2 + (2\eta^2L^2 + 2\eta \mu) \|\z^{(t)}-\w^{(t)}\|^2 - \|\z^{(t)}-\w^{(t)}\|^2.
	\end{align*}
	By setting $\eta= \frac{1}{4L}$, we have $2\eta^2L^2 + 2\eta \mu \leq 1$, and then 
	\begin{align*}
	\| \w^{(t+1)} - \w_* \|^2&\leq (1-\eta\mu) \|\w^{(t)}-\w_*\|^2\leq \left(1-\frac{\mu}{4L}\right)^{t+1}\|\w^{(0)}-\w_*\|^2.
	\end{align*}
	The conclusion follows by applying Lemma~\ref{lem:prom}.
	
\end{proof}

%\begin{corollary}\label{thm:eg}
%	Suppose Assumption~\ref{ass:3} and Assumption~\ref{assume:additional} holds,  and Algorithm~\ref{alg:EG} is used to implement $\text{ApproxSVI}$. Algorithm~\ref{alg:meta} with $\gamma=1/(2\rho)$, $\theta_k =(k+1)^\alpha$ with $\alpha>1$,  $ T_k= 8L/\rho\log(4(k+1)L/\rho)$, and a total of stages $K=6D^2(\alpha+1)/\epsilon^2$  guarantees
%	\begin{align}\label{eqn:fc3}%\label{eqn:proximal mapping_converge}
%	\E[ \|\z_\tau - \bar\z_\tau \|^2]\leq \epsilon^2,%\leq \frac{2D^2\sqrt{K}}{(2/3)K^{3/2}} + \frac{4\sum_{k=0}^{K-1}\sqrt{k+1}\varepsilon(\eta_k, T_k)}{(2/3)\rho K^{3/2}}.
%	\end{align}
%	and
%	\begin{align}\label{eqn:fc4}%\label{eqn:proximal mapping_converge}
%	\E[\max_{\z\in\Z}\langle F(\bar \z_\tau), \bar\z_\tau - \bz\rangle]\leq  2D\rho\epsilon%\leq \frac{2D^2\sqrt{K}}{(2/3)K^{3/2}} + \frac{4\sum_{k=0}^{K-1}\sqrt{k+1}\varepsilon(\eta_k, T_k)}{(2/3)\rho K^{3/2}}.
%	\end{align}
%	where $\bar\z_\tau$ is the solution to $\text{SVI}(F^\gamma_{\z_\tau}, \Z)$. The total iteration complexity is $\widetilde O(D^2L/(\rho\epsilon^2))$.
%\end{corollary}
%{\bf Remark: } The total iteration complexity for finding a nearly $\epsilon$-stationary solution of the corresponding min-max saddle-point problem  is then $\widetilde O(L\rho/\epsilon^2)$, which is similar to that using the Nesterov's accelerated method for solving the strongly monotone subproblems.
%	
%	

\section{Other Methods for Solving Strongly Monotone SVI with a Lipschitz Continuous Mapping}
%We first present an assumption and a useful lemma that will be used in this section.
%\begin{assumption}~\label{assume:additional}
%$F(\bz)$ is single-valued and $L$-Lipschitz continuous.
%\label{assume:bounded}
%At least one of the statements below hold
%	\begin{itemize}
%		\item[A:]  For any $\bz\in\Z$, there exists a stochastic oracle that returns a random vector $\bxi(\z)$  such that $\E[\bxi(\z)]\in F(\bz)$ and $\E\|\bxi(\z)\|^2\leq G^2$ for a constant $G$. 
%\label{assume:Lipcont}
%		\item[B:] 
%		$F(\bz)$ is single-valued and $L$-Lipschitz continuous.
%	\end{itemize}	
%\end{assumption}
\label{sec:otherfom}
%We first present an useful lemma in this section. 
%\begin{lemma}\label{lem:prom}
%	Suppose that there exists an algorithm for a monotone $\text{SVI}(F, \Z)$ with $L$-Lipschitz continuous single-valued mapping $ F(\z)$ that returns a solution $\zh$. Then by constructing $\bar\z = \text{Proj}(\zh - \eta F(\zh))$ with $\eta= 1/(\sqrt{2}L)$, we have
%	\begin{align}\label{eqn:key}
%	\max_{\z\in\Z}F(\bar\z)^{\top}(\bar\z - \z)\leq DL(2+ \sqrt{2})\|\zh - \w_*\|,
%	\end{align}
%	where $\w_*$ is the solution of $\text{SVI}(F, \Z)$.
%\end{lemma}

\subsection{Using the Nesterov's Accelerated Method}We first present the Nesterov's accelerated method~\cite{1078-0947Nesterov} for solving strongly monotone SVI in Algorithm~\ref{alg:NA} and show that it could achieve smaller complexity than the GD method when the condition number $L/\rho$ is large. 
\begin{proposition}\label{prop:na}When $F(\bz)$ is single-valued and $L$-Lipschitz continuous.
	and  is $\mu$-strongly monotone,  Algorithm~\ref{alg:NA} guarantees that for any $\z\in\Z$
	\begin{eqnarray}\label{eqn:ineq2}
	\max_{\z\in\Z}F(\bar \z)^{\top}(\bar\z-\z)\leq 6\sqrt{M} DL^2/\mu^{3/2}\exp(-T/2(\beta+1))
	\end{eqnarray}
	where $\beta = L/\mu$. 
	In addition, we have
	\[
	\mu \|\bar\z - \w_*\|^2 \leq 6\sqrt{M} DL^2/\mu^{3/2}\exp(-T/2(\beta+1))
	\]
	where $\w_*$ denotes a solution to $\text{SVI}(F, \Z)$, and $M = \max_{\z, \w\in\Z}F(\w)^{\top}(\z - \w) + \frac{\mu}{2}\|\z -\w\|^2$.
	
\end{proposition}
\begin{proof}
	The proof is following. First, according to Theorem 3 in~\cite{1078-0947Nesterov}, we have
	\begin{align*}
	\|\zh_T - \w_*\|^2\leq 2M\beta^2/\mu \exp(-T/(\beta+1))
	\end{align*}
	Following Lemma~\ref{lem:prom}, we have
	\begin{align*}
	\max_{\z\in\Z}F(\bar\z)^{\top}(\bar\z - \z) &\leq DL(2+ \sqrt{2})\|\zh_T - \w_*\|\\
	&\leq 6\sqrt{M} DL^2/\mu^{3/2}\exp(-T/2(\beta+1))\\
	\end{align*}
	
\end{proof}
\begin{algorithm}[tb]
	\caption{Nesterov's Accelerated  Method for $\text{SVI}(F, Z)$: NA$(F,\Z,\bw^{(0)}, T)$}\label{alg:NA}
	\begin{algorithmic}[1]
		\STATE \textbf{Input:} $\mu-$Strongly Monotone and $L$-Lipschitz continuous Mapping $F$, set $\Z$, $\bw^{(0)}\in\mathcal{Z}$,  and an integer $T\geq 1$.  
		\STATE Initialize $\lambda_0= 1$ and $S_0 = 1$, $\beta = L/\mu$, $\eta = 1/(\sqrt{2}L)$
		\FOR{$t=0,..., T-1$}
		\STATE $\w^{(t)}= \arg\max_{\x\in\Z}\sum_{k=0}^t\lambda_k(F(\z^{(k)})^{\top}(\z^{(k)} - \x) - \frac{\mu}{2}\|\x - \z^{(k)}\|^2)$ 
		\STATE $\z^{(t+1)}= \arg\max_{\x\in\Z}F(\w^{(k)})^{\top}(\w^{(t)} - \x) - \frac{L}{2}\|\x - \w^{(t)}\|^2$ 
		%\STATE $S_k = S_{k-1} + \lambda_k$
		\STATE $\lambda_{t+1} = \frac{S_t}{\beta}$
		\STATE $S_{t+1} = S_{t} + \lambda_{t+1}$
		\ENDFOR
		\STATE Compute $\zh_T =\frac{1}{S_T}\sum_{t=0}^{T}\lambda_t\z^{(t)}$
		\STATE \textbf{Output:} $ \bar\z= \text{Proj}_{\Z}\left(\zh_T - \eta F(\zh_T)\right)$
	\end{algorithmic}
\end{algorithm}

\begin{corollary}\label{thm:na}
	Suppose Assumption~\ref{ass:3}  holds and $F(\bz)$ is single-valued and $L$-Lipschitz continuous,  and Algorithm~\ref{alg:NA} is used to implement $\text{ApproxSVI}$. Algorithm~\ref{alg:meta} with $\gamma=1/(2\rho)$, $\theta_k =(k+1)^\alpha$ with $\alpha>1$,  $ T_k= 2(L/\rho+1)\log(6(k+1)L^{2}M^{1/2}/(D\rho^{5/2}))$, and a total of stages $K=6D^2(\alpha+1)/\epsilon^2$  guarantees
	\begin{align}\label{eqn:fc3}%\label{eqn:proximal mapping_converge}
	\E[ \|\z_\tau - \bar\z_\tau \|^2]\leq \epsilon^2,%\leq \frac{2D^2\sqrt{K}}{(2/3)K^{3/2}} + \frac{4\sum_{k=0}^{K-1}\sqrt{k+1}\varepsilon(\eta_k, T_k)}{(2/3)\rho K^{3/2}}.
	\end{align}
	and
	\begin{align}\label{eqn:fc4}%\label{eqn:proximal mapping_converge}
	\E[\text{dist}^2(0, \partial (f(\bar\u_\tau, \bar\v_\tau)+ 1_{\Z}(\bar\u_\tau, \bar\v_\tau))]\leq \epsilon^2/\gamma^2
	%\E[\max_{\z\in\Z}\langle F(\bar \z_\tau), \bar\z_\tau - \bz\rangle]\leq  2D\rho\epsilon%\leq \frac{2D^2\sqrt{K}}{(2/3)K^{3/2}} + \frac{4\sum_{k=0}^{K-1}\sqrt{k+1}\varepsilon(\eta_k, T_k)}{(2/3)\rho K^{3/2}}.
	\end{align}
	where $\bar\z_\tau$ is the solution to $\text{SVI}(F^\gamma_{\z_\tau}, \Z)$. The total iteration complexity is $\widetilde O(D^2L/(\rho\epsilon^2))$.
\end{corollary}

{\bf Remark: } The total iteration complexity for finding a nearly $\epsilon$-stationary solution of the corresponding min-max saddle-point problem  such that $\E[ \|\z_\tau - \bar\z_\tau \|^2]\leq (\gamma\epsilon)^2$  is then $\widetilde O(L\rho/\epsilon^2)$. By comparing to the result in Corollary~\ref{thm:gd}, the above result of using Nesterov's accelerated method for solving the strongly monotone subproblems is better by a factor of $L/\rho$.  %By applying the above result to a SVI with $L$-Lipschitz continuous single-valued mapping, the  iteration complexity for ensuring  $\E[\max_{\z\in\Z}\langle F(\bar \z_\tau), \bar\z_\tau - \bz\rangle]\leq\epsilon$ is $O(L\rho D^4/\epsilon^2\log(1/\epsilon))$, which is better than the iteration complexity of $O(L^2D^4/\epsilon^2)$ reported by~\cite{DBLP:journals/coap/DangL15} when the condition number $L/\rho\gg 1$ is large. 

\section{Solving Weakly Monotone SVI}\label{sec:last}
As we mentioned in the introduction, the proposed algorithms can be used for solving a more general SVI problem. We present and discuss the results in this section. 
Recall the SVI problem defined in \eqref{eq:SVI} and the MVI problem defined  in \eqref{eq:MVI}, where $F$ is not necessarily pertained to any $f(\x, \y)$. 
In the literature of VI~\cite{Nemirovski2015InformationbasedCO}, a solution $\z^*$ that satisfies~(\ref{eq:SVI}) is also called strong solution, and a solution $\z_*$ that satisfies~(\ref{eq:MVI}) is called a weak solution. When $F$ is monotone, finding a solution for $\text{SVI}(F,\Z)$ is typically a tractable problem~\cite{Nemirovski2015InformationbasedCO}. %So is finding a solution for $\text{MVI}(F,\Z)$ because, when $F$ is monotone, it can be showed by definition that a solution of $\text{SVI}(F,\Z)$ is also a solution of $\text{MVI}(F,\Z)$~\cite{Nemirovski2015InformationbasedCO}. 
In order to make a SVI problem with a non-monotone set-valued mapping $F$ tractable, we impose the same assumptions as in Assumption~\ref{ass:3}. Similar (or stronger) assumptions have been used in previous studies for solving non-monotone SVI. %following assumption is made throughout the paper and is critical to establish all results in this  %For a non-monotone mapping $F$,  it is proved that a solution of $\text{MVI}(F,\mathcal{Z})$ is also a solution of $\text{SVI}(F,\mathcal{Z})$ whenever $F$ is continuous~\cite{}. 

When applying an iterative numerical algorithm to solve $\text{SVI}(F,\Z)$, it is generally hard to guarantee an exact solution for $\text{SVI}(F,\Z)$ after a finite number of iterations. Therefore, an alternative goal is to find an $\epsilon$-\emph{gap solution} for $\text{SVI}(F,\Z)$, namely, a solution $\bar\bz$ such that 
\begin{align}\label{eqn:epsilonstr}
\exists \bar\bxi\in F(\bar\bz)\text{ s.t. }\max_{\bz\in\Z}\langle \bar\bxi, \bar\bz - \bz\rangle\leq \epsilon.%,%~\forall  \bz\in\mathcal{Z}.
\end{align}
%It is also easy to prove with monotonicity that an $\epsilon$-solution for $\text{SVI}(F)$ is also an  $\epsilon$-\emph{solution} for $\text{MVI}(F)$, which is a solution $\bar\bz$ such that 
%$$
%\left\langle \bxi, \bar\bz - \bz\right\rangle\leq \epsilon,~\forall  \bz\in\mathcal{Z},\forall  \bxi\in F(\bz).
%$$
However, without additional assumption on $F$, finding an $\epsilon$-\emph{gap solution} for $\text{SVI}(F,\Z)$ in finite iterations can be also challenging even if $F$ is monotone. For example, consider the SVI problem of finding $z^*\in[-1,1]$ such that $\left\langle \xi^*, z-z^*\right\rangle\geq 0$ for some $\xi^*\in \partial|z^*|$ and all $z\in[-1,1]$, which is associated to the convex minimization $\min_{z\in[-1,1]}|z|$ and has a solution at $0$. Hence,  if $\bar z$ is very close to $0$ but not $0$, we always have $\langle \bar\xi, \bar z - z\rangle\geq 1$ and $|\bar\xi|=1$ for any $\bar\xi\in\partial|\bar z|$ and $z=-\text{sign}(\bar z)$. 
To address this issue, we introduce the notion of nearly $(\epsilon,\delta)$-gap solution to $\text{SVI}(F,\Z)$. 
\begin{definition}
	\label{def:nearlygap}
	A point $\w\in\Z$ is called a nearly $(\epsilon,\delta)$-gap solution to $\text{SVI}(F,\Z)$ for $\epsilon>0$ and $\delta>0$ if there exists $\bar\w\in\Z$ and such that 
	\begin{align*}
	\|\w - \bar\w\|\leq \delta, \quad \exists \bar\bxi\in F(\bar\w)\text{ s.t. }\max_{\bz\in\Z}\langle \bar\bxi, \bar\w - \bz\rangle\leq \epsilon.
	\end{align*}
\end{definition}
%\vspace*{-0.1in}
If $\delta = O(\epsilon)$, we simply call this solution as nearly $\epsilon$-gap solution. 
%In order to show the existence of nearly $(\epsilon,\delta)$-gap solutions, we define a \emph{proximal-point mapping} (PPM) of $F$ as 
%\begin{eqnarray}
%\label{eq:ppm2}
%F_{\bw}^\gamma(\bz)\equiv F(\bz)+\frac{1}{\gamma}(\bz-\bw)
%\end{eqnarray} 
%for $\bw\in\Z$ and $0<\gamma<\rho^{-1}$. According to Lemma~\ref{lem:PPM}, $F_{\bw}^\gamma$ is $(\frac{1}{\gamma}-\rho)$-strongly monotone so that $\text{SVI}(F_{\bw}^\gamma,\Z)$ has a unique solution denoted by $\widehat\bw$.  	%We define a finite constant $D$ that satisfies
%%$$
%%D\equiv \max_{\bz,\bz'\in\Z}\|\bz-\bz'\|.
%%D^2\geq 2\max_{\bz,\bz'\in\Z}V(\bz,\bz').
%%$$
The lemma below implies that the proposed algorithms can find a nearly $\epsilon$-gap solution for a SVI problem.  %characterizes the relationship between $\widehat\bw$ and $\bw$ with its proof given in Section~\ref{sec:proofslemmas} in Appendix.
\begin{lemma}\label{lem:zbw}
	Let $F_{\bw}^\gamma$ be defined in \eqref{eq:ppm} for $0<\gamma<\rho^{-1}$ and $\bw\in\Z$ and $\bar\bw$ be the solution to $\text{SVI}(F_{\bw}^\gamma,\Z)$.  We have 
	\begin{eqnarray}
	\label{eq:zbw}
	\exists \bar\bxi\in F(\bar\bw)\text{ s.t. }\max_{\z\in\Z}\langle \bar\bxi, \bar\bw - \bz\rangle\leq  \frac{D}{\gamma}\|\bw-\bar\bw\|.% ,~\forall  \bz\in\mathcal{Z}.
	\end{eqnarray} 
\end{lemma}

According to this lemma, if we can find a solution $\bw\in\Z$ such that $\|\bw-\bar\bw\|\leq\frac{\gamma\epsilon}{D}$, we will have 
$\max_{\z\in\Z}\langle \bar\bxi,  \bar\bw - \bz \rangle\leq \epsilon$,  namely, $\bw$ is $(\gamma\epsilon/D)$-closed to an $\epsilon$-gap solution of $\text{SVI}(F,\Z)$. %That's why we call $\bw$  a nearly $(\epsilon,\frac{\gamma\epsilon}{D})$-gap solution.

Next we present a proposition to show that, when $F(\z)$ is single-valued and Lipschitz continuous,  a nearly $(\epsilon,\delta)$-gap solution with $\delta=O(\epsilon)$ is an $O(\epsilon)$-gap solution to $\text{SVI}(F, \Z)$, and an $\epsilon$-gap solution $\w$ is $O(\sqrt{\epsilon})$-close to the solution of $\text{SVI}(F_{\w}^\gamma, \Z)$ for $\gamma\in(0,L^{-1})$. 
\begin{proposition}\label{prop:convert}
	When $F(\z)$ is single-valued and $L$-Lipschitz continuous, the following statements hold:
	%\vspace*{-0.1in}
	\begin{itemize}[leftmargin=*]
		\item If $\w$ is a nearly $(\epsilon,c\epsilon)$-gap solution for $c>0$, then  $\w$ is an $(\left(1+LcD+Mc \right)\epsilon)$-gap solution  to $\text{SVI}(F, \Z)$, where $M\equiv \max_{\z\in\Z}\|F(\z)\|$.
		%where $\gamma$ is the same as in Definition~\ref{def:nearlygap}.
		\item If $\w$ is an $\epsilon$-gap solution to $\text{SVI}(F, \Z)$, then $\|\w - \wh\|\leq \sqrt{\frac{\epsilon}{\gamma^{-1} - L}}$, where $\wh$ is the unique solution of $\text{SVI}(F_{\w}^\gamma, \Z)$ for $0<\gamma<L$. 
	\end{itemize}
\end{proposition}

%{\bf Remark:} We can see that proving  $\|\w - \wh\|\leq O(\epsilon)$ with $\wh$ being the solution of $\text{SVI}(F_{\w}^\gamma, \Z)$ is a more general approach that can cover both Lipschitz continuous single-valued mappings and non-Lipschitz continuous set-valued mappings. In Section~\ref{sec:5}, we will further show that $\|\w - \wh\|\leq O(\epsilon)$  implies that $\w$ is a nearly $\epsilon$-stationary solution to the corresponding min-max saddle-point problem. Moreover, the second part of Proposition~\ref{prop:convert} and the discussion presented in Section~\ref{sec:5} imply that an $\epsilon$-gap solution can only lead to a nearly $\sqrt{\epsilon}$-stationary solution, which is worse than directly proving $\|\w - \wh\|\leq O(\epsilon)$. 
%\vspace*{-0.1in}
This proposition indicates that a nearly $\epsilon$-gap solution  is the right target when solving $\text{SVI}(F, \Z)$ no matter $F$ is single-valued Lipschitz continuous or set-valued non-Lipschitz. %In Section~\ref{sec:5}, we will further show that  a $(\epsilon,\delta)$-gap solution of $\text{SVI}(F, \Z)$ with $\delta=O(\epsilon)$ is a nearly $O(\epsilon)$-stationary solution to the corresponding min-max problem. Moreover, the second part of Proposition~\ref{prop:convert} and the discussion presented in Section~\ref{sec:5} imply that an $\epsilon$-gap solution can only lead to a nearly $\sqrt{\epsilon}$-stationary solution of the min-max problem, which is worse than the guarantee by a $(\epsilon,\delta)$-gap solution to $\text{SVI}(F, \Z)$. 
% with $\delta=O(\epsilon)$. 

Now consider the second statement. Suppose $\bar \bz\in\Z$ is an $\epsilon$-gap solution of $\text{SVI}(F,\Z)$ such that 
$\max_{\bz\in\Z}\langle F(\bar \bz), \bar\bz - \bz\rangle\leq \epsilon.$ Note that $F_{\bar\bz}^\gamma$ is $\left(\frac{1}{\gamma}-L\right)$-strongly monotone for $\gamma\in(0,L^{-1})$ because $F$ is $L$-Lipschitz continuous. Let $\widehat\bz$ be the unique solution of $\text{SVI}(F_{\bar\bz}^\gamma,\Z)$ for $\gamma\in(0,L^{-1})$. We then have 
$$
\left\langle F(\widehat\bz)+\frac{1}{\gamma}(\widehat\bz-\bar\bz),\widehat\bz-\bar\bz\right\rangle\leq 0
$$
which, by the Lipschitz continuity of $F$, implies
$$
\frac{1}{\gamma}\|\widehat\bz-\bar\bz\|^2\leq 
\left\langle F(\widehat\bz),\bar\bz-\widehat\bz\right\rangle
\leq \left\langle F(\bar\bz),\bar\bz-\widehat\bz\right\rangle+L\|\bar\bz-\widehat\bz\|^2
\leq \epsilon+L\|\bar\bz-\widehat\bz\|^2.
$$
By reorganizing terms, we have $\left(\frac{1}{\gamma}-L\right)\|\bar\bz-\widehat\bz\|^2\leq \epsilon$ which leads to the conclusion of the second statement.

The next two corollaries summarize the convergence results of Algorithm~\ref{alg:meta} for the solving the SVI problem~(\ref{eq:SVI}). 
\begin{corollary}\label{thm:sgd-2}
	For the SVI problem~(\ref{eq:SVI}), assume that Assumption~\ref{ass:3} holds and for any $\z\in\Z$ there exists $\bzeta(\z)\in F(\z)$ such that  $\E[\bzeta(\z)]\in F(\z)$, and $\E\|\bzeta(\z)\|^2\leq G^2$. Under the same conditions as in Corollary~\ref{thm:sgd} except for   $K=\frac{(16\rho^2D^2+4\rho DG)D^2(\alpha+1)}{\epsilon^2}$  in Algorithm~\ref{alg:meta}, we have
	\begin{align*}%\label{eqn:fc}%\label{eqn:proximal mapping_converge}
	&\E[ \|\z_\tau - \bar\z_\tau \|^2]\leq \frac{\gamma^2\epsilon^2}{D^2},\\%\leq \frac{2D^2\sqrt{K}}{(2/3)K^{3/2}} + \frac{4\sum_{k=0}^{K-1}\sqrt{k+1}\varepsilon(\eta_k, T_k)}{(2/3)\rho K^{3/2}}.
	&	\exists \bxi\in F(\bar\z_\tau)\text{ s.t. }\E[\max_{\z\in\Z}\langle \bxi, \bar\z_\tau - \bz\rangle]\leq  \epsilon,%\leq \frac{2D^2\sqrt{K}}{(2/3)K^{3/2}} + \frac{4\sum_{k=0}^{K-1}\sqrt{k+1}\varepsilon(\eta_k, T_k)}{(2/3)\rho K^{3/2}}.
	\end{align*}
	where $\bar\z_\tau$ is the solution to $\text{SVI}(F^\gamma_{\z_\tau}, \Z)$. 
	%	 and
	%	\begin{align}\label{eqn:fc}%\label{eqn:proximal mapping_converge}
	%	\exists \bxi\in F(\bar\z_\tau)\text{ s.t. }\E[\max_{\z\in\Z}\langle \bxi, \bar\z_\tau - \bz\rangle]\leq  \epsilon,%\leq \frac{2D^2\sqrt{K}}{(2/3)K^{3/2}} + \frac{4\sum_{k=0}^{K-1}\sqrt{k+1}\varepsilon(\eta_k, T_k)}{(2/3)\rho K^{3/2}}.
	%	\end{align}
	%%	
	The  total iteration complexity of $O(\frac{1}{\epsilon^6})$.
	%$\sum_{k=0}^{K-1}T_k=O(1/\epsilon^6)$. 
\end{corollary}
%\vspace*{-0.1in}
{\bf Remark: } The above result establishes the convergence result for finding a nearly $(\epsilon,\frac{\gamma\epsilon}{D})$-gap solution for $\text{SVI}(F, \Z)$. The total iteration complexity can be easily derived as $\sum_{k=1}^Kk^2 = O(1/\epsilon^6)$. To our knowledge, this is the first non-asymptotic convergence of stochastic algorithms for solving SVI without the monotone and Lipschitz conditions. 
\begin{corollary}\label{thm:gd-2}
	For the SVI problem~(\ref{eq:SVI}), assume that Assumption~\ref{ass:3} holds  and  $F$ is single-valued and $L_k$-Lipschitz continuous.  Under the same conditions as in Corollary~\ref{thm:gd} except for   $K= \frac{16\rho^2 D^4(\alpha+1)}{\epsilon^2}$ in Algorithm~\ref{alg:meta}, we have
	\begin{align*}%\label{eqn:fc}%\label{eqn:proximal mapping_converge}
	&\E[ \|\z_\tau - \bar\z_\tau \|^2]\leq \frac{\gamma^2\epsilon^2}{D^2},\quad%\leq \frac{2D^2\sqrt{K}}{(2/3)K^{3/2}} + \frac{4\sum_{k=0}^{K-1}\sqrt{k+1}\varepsilon(\eta_k, T_k)}{(2/3)\rho K^{3/2}}.
	\E[\max_{\z\in\Z}\langle F(\bar \z_\tau), \bar\z_\tau - \bz\rangle]\leq  \epsilon%\leq \frac{2D^2\sqrt{K}}{(2/3)K^{3/2}} + \frac{4\sum_{k=0}^{K-1}\sqrt{k+1}\varepsilon(\eta_k, T_k)}{(2/3)\rho K^{3/2}}.
	\end{align*}
	where $\bar\z_\tau$ is the solution to $\text{SVI}(F^\gamma_{\z_\tau}, \Z)$. 
	The total iteration complexity of $O(\log(\frac{1}{\epsilon})\frac{D^4L^2}{\epsilon^2})$.
\end{corollary}
{\bf Remark: }One can similarly derive the results based on the Nesterov's accelerated method~\cite{1078-0947Nesterov}  and the extragradient method~\cite{10017556617}, which can improve the above complexity when the condition number $L/\rho\gg 1$ is large. Next, we compare the above result with~\cite{DBLP:journals/coap/DangL15} for solving  a SVI with $L$-Lipschitz continuous single-valued mapping, which presents an extragradient algorithm and needs to compute two gradient updates at each iteration and has an iteration complexity of $O(L^2D^4/\epsilon^2)$ for finding an $\epsilon$-gap solution to the SVI.  Without additional assumption on the value of $\rho$, we can set $\rho=L$. Then the  iteration complexity of Algorithm~\ref{alg:meta} for ensuring  $\E[\|\z_\tau - \bar\z_\tau\|]\leq \epsilon/(LD)$ and $\E[\max_{\z\in\Z}\langle F(\bar \z_\tau), \bar\z_\tau - \bz\rangle]\leq 2\epsilon$ is $O(L^2D^4/\epsilon^2\log(1/\epsilon))$, which implies that $\E[\max_{\z\in\Z}\langle F(\z_\tau), \z_\tau - \bz\rangle]\leq O(\epsilon)$ according to Proposition~\ref{prop:convert}. As we can see our complexity is worse by a $\log(1/\epsilon)$ factor but only needs to perform one gradient update at each iteration.

{\bf Discussion: $\epsilon$-gap vs $\epsilon$-stationary.} %One might be  interested in whether one can apply an algorithm (e.g., the one presented in~\citet{DBLP:journals/coap/DangL15} ) for finding a nearly $\epsilon$-stationary solution for the min-max saddle-point problem~(\ref{eqn:P}) that finds an $\epsilon$-gap solution for the corresponding SVI problem. % Before ending this section, we present some discussions on the complexity of an algorith
Recall the definition of an $\epsilon$-stationary solution: 
\begin{align}\label{eq:minmaxstationary}
%\text{dist}^2(\mathbf{0}, \partial_x [f(\x, \y)+ 1_{\X}(\x)]\times \partial_y [-f(\x, \y) + 1_{\Y}(\y)])\\
\text{dist}(\mathbf{0}, \partial (f(\x, \y)+ 1_{\Z}(\x,\y))\leq \epsilon,
\end{align}
It is an interesting question whether one can derive an $\epsilon$-stationary result for the min-max problem (i.e., \eqref{eq:minmaxstationary}) from an $\epsilon$-gap solution that satisfies \eqref{eqn:epsilonstr} of the corresponding SVI. We will show that \eqref{eq:minmaxstationary} is a stronger result than  \eqref{eqn:epsilonstr}, i.e., $\epsilon$-stationary solution is also an $O(\epsilon)$-gap solution of the corresponding SVI but not vice versa. %To compare the strength of \eqref{eq:minmaxstationary} and \eqref{eqn:epsilonstr}, we present the following observations.
Let $F(\bz) = (\partial_x f(\x, \y), - \partial_y f(\x, \y))^{\top}$ and $\Z=\X\times\Y$. Suppose a solution $\bar\bz$ is found such that $\text{dist}^2(0, F(\bar\bz)+ 1_{\Z}(\bar\bz))\leq \epsilon^2$. Hence, there exists $\bar\bxi\in F(\bar\bz)$ and $\bar\bzt\in \partial 1_{\Z}(\bar\bz)$ such that $\|\bar\bxi+\bar\bzt\|\leq \epsilon$. Note that $\bar\bzt$ is a vector in the normal cone of $\Z$ at $\bar\bz$ so that $\langle \bar\bzt, \bar\bz - \bz\rangle\geq 0$ for any $\bz\in\Z$. Hence, we can easily show that 
$$
\max_{\bz\in\Z}\langle \bar\bxi, \bar\bz - \bz\rangle\leq \max_{\bz\in\Z}\langle \bar\bxi+\bar\bzt, \bar\bz - \bz\rangle\leq \|\bar\bxi+\bar\bzt\|D\leq \epsilon D.
$$
This means an $\epsilon$-stationary solution is an $\epsilon D$-gap solution for the corresponding SVI.  However, the reversed direction is not true. We consider the following problem in $\mathbb{R}^2\times\mathbb{R}$ 
$
\min_{x_1^2+x_2^2\leq r^2}\max_{y\in[-r,r]} x_1
$
where the objective function is also viewed as a (constant) function of $y$. Consider the solution $\bar\bz=(\bar x_1,\bar x_2,\bar y)=(0,r,0)$ which is on the boundary of the feasible region and corresponds to $F(\bar\bz)+ 1_{\Z}(\bar\bz)=\{\bxi\in\mathbb{R}^3|\xi_1=1,\xi_2\geq 0,\xi_3=0\}$. As a result, $\text{dist}^2(0, \partial (f(\bar\x, \bar\y)+ 1_{\Z}(\bar\x, \bar\y))=1$ but, because $F(\bar\bz)=(1,0,0)$, we have
$\max_{\bz\in\Z}\langle F(\bar\bz), \bar\bz - \bz\rangle=r$. Then, if $r$ is small, we have $\bz$ satisfying \eqref{eqn:epsilonstr}. This means \eqref{eq:minmaxstationary} is stronger than \eqref{eqn:epsilonstr}.

\subsection{Proof of Lemma~\ref{lem:zbw}}
\begin{proof}[Proof of Lemma~\ref{lem:zbw}]
	Since $\bar\bw$ is the solution to $\text{SVI}(F_{\bw}^\gamma,\Z)$, there exists $\bar\bxi\in F(\bar\bw)$ such that 
	\small
	\begin{eqnarray*}
		\left\langle \bar\bxi+\frac{1}{\gamma}(\bar\bw-\bw), \bz-\bar\bw\right\rangle\geq 0,
	\end{eqnarray*}
	\normalsize
	for any $\bz\in\Z$. The conclusion is proved by reorganizing terms and using the fact that $\|\bz-\bar\bw\|\leq  D$.
	%$\|\bz-\widehat\bw\|\leq \sqrt{2V(\bz,\widehat\bw)} \leq D$
\end{proof}

\subsection{Proof of Proposition~\ref{prop:convert}}
\label{sec:prop:convert}
Consider the first statement. Suppose $\w\in \Z$ is a nearly $(\epsilon,c\epsilon)$-gap solution $\text{SVI}(F,\Z)$. By its definition, there exists $\wh\in \Z$ such that  $\|\w - \wh\|\leq c\epsilon$ and $\max_{\bz\in\Z}\langle F(\wh), \wh - \bz\rangle\leq \epsilon$.
% where $\wh$ is the unique solution of $\text{SVI}(F_{\w}^\gamma,\Z)$.
%and $\bar\bz\in\Z$ where $\bar\z$ is a solution of $\text{SVI}(F_{\w}^\gamma,\Z)$ such that  
Since $F(\z)$ is Lipschitz continuous and $\Z$ is compact, $M$ is finite. 
%there exists $G>0$ such that $\max_{\z\in\Z}\|F(\z)\|\leq M$. 
Then for any $\z\in\Z$ we have
\begin{align*}
\langle F( \w), \w - \bz\rangle&\leq \langle F(\wh), \wh - \bz\rangle + \langle  F( \w) - F(\wh), \wh - \z\rangle +\langle  F( \w), \w - \wh\rangle\\
&\leq \epsilon + L\| \w-\wh\|\| \wh-\bz\| + M\|\w-\wh\|\\
&\leq \left(1+LcD+Mc \right)\epsilon
\end{align*}

%we will show that our algorithms are likely to have better complexity for finding a nearly $\epsilon$-stationary solution for the corresponding min-max saddle-point problems than~\citet{DBLP:journals/coap/DangL15}.  

In light of the above discussion, let us compare the result of applying the extragradient algorithm analyzed in~\citep{DBLP:journals/coap/DangL15} to the min-max problem~(\ref{eqn:P}) when $f(\x, \y)$ is smooth both in $\x$ and $\y$ such that the corresponding $F(\z)$ is Lipschitz continuous. Their result is that finding a $\bar\z$ such that $\max_{\bz\in\Z}\langle F(\bar\z), \bar\bz - \bz\rangle\leq \epsilon$ requires a complexity of $O(1/\epsilon^2)$. According to Proposition~\ref{prop:convert}, this implies that for finding a nearly $\epsilon$-stationary solution to the min-max problem, the complexity of the extragradient algorithm analyzed in~\citep{DBLP:journals/coap/DangL15} is $O(1/\epsilon^4)$. In contrast, our complexity in Corollary~\ref{thm:gd-2} is $O(\log(1/\epsilon)/\epsilon^2)$. It is worth mentioning that it is unclear whether an improved analysis of the extragradient method as in~\citep{DBLP:journals/coap/DangL15} can have a better complexity than $O(1/\epsilon^4)$ for finding a nearly $\epsilon$-stationary solution for a min-max saddle-point problem.

\vskip 0.2in
%\bibliography{iclr2019_conference,all}

\bibliographystyle{plain}

\end{document}